\documentclass[11pt]{amsart}

\usepackage{DKstyle}
\usepackage{bm}
\usepackage{float}
 \usepackage[active]{srcltx}

\def\scrY{{\mathcal Y}}
\newcommand{\col}{\: : \:}
\def\DI{\operatorname{\mathbf{DI}}}

\newcommand{\vertiii}[1]{{\left\vert\kern-0.25ex\left\vert\kern-0.25ex\left\vert #1
    \right\vert\kern-0.25ex\right\vert\kern-0.25ex\right\vert}}
 \makeatletter
\newcommand*{\rom}[1]{\expandafter\@slowromancap\romannumeral #1@}
\makeatother
\theoremstyle{plain}

\newcommand{\fg}{\mathfrak{g}}
\renewcommand{\fV}{\mathfrak{V}}

\usepackage{caption}
\usepackage[labelfont=rm]{subcaption}

\usepackage[pagebackref=false, colorlinks]{hyperref}

\hypersetup{pdffitwindow=true,linkcolor=blue,citecolor=blue,urlcolor=blue,menucolor=blue}

\usepackage{comment}

\subjclass{}%
\keywords{}%

\date{\today}%
\thanks{  SY was supported by the
National Key R\&D Program of China No. 2024YFA1015100. 
AS was supported by the Knut and Alice Wallenberg Foundation and also by the Swedish Research Council Grant 2023-03411.
}

\dedicatory{}%
\commby{}%

\title{A zero-one law for improvements to Dirichlet's theorem in arbitrary dimension}
\author[Andreas Str\"ombergsson]{Andreas Str\"ombergsson}
\address{Andreas Str\"ombergsson, Department of Mathematics, Uppsala University, Box 480, SE-75106, Uppsala, Sweden
\newline \rule[0ex]{0ex}{0ex} \hspace{8pt}{\tt astrombe@math.uu.se}}

\author[Shucheng Yu]{Shucheng Yu}
\address{Shucheng Yu, School of Mathematical Sciences, University of Science and Technology of China (USTC), 230026, Hefei, China
\newline \rule[0ex]{0ex}{0ex} \hspace{8pt}{\tt yusc@ustc.edu.cn}}
\thanks{}
\begin{document}
\begin{abstract}
Let $\psi$ be a continuous decreasing function defined for all sufficiently large positive real numbers. %
We say that a real
 $m\times n$ matrix $A$ is $\psi$-Dirichlet if for every sufficiently large real number $t$, one can find $\bm{p} \in \Z^m$, $\bm{q} \in \Z^n\setminus\{\bm{0}\}$ satisfying $\|A\bm{q}-\bm{p}\|^m< \psi({t})$ and $\|\bm{q}\|^n<{t}$. %
By removing a technical condition from a partial zero-one law proved in \cite{KleinbockStrombergssonYu2022}, we prove a zero-one law for the Lebesgue measure of the set of $\psi$-Dirichlet matrices provided that 
$\psi(t)<1/t$ and  
$t\psi(t)$ is increasing. 
In fact, we prove the zero-one law in a more general situation with the
monotonicity assumption on $t\psi(t)$ replaced by a weaker condition.
Our proof follows the dynamical approach in \cite{KleinbockStrombergssonYu2022} by reducing the question to a shrinking target problem in the space of lattices.
 The key new ingredient is a family of carefully 
  chosen subsets of the shrinking targets studied in \cite{KleinbockStrombergssonYu2022}, together with a geometric disjointness argument controlling the short-range correlations for the associated hitting events.
 Our method also works for the analogous weighted
problem where the relevant supremum norms are replaced by certain weighted quasi-norms.
\end{abstract}

\maketitle

\setcounter{tocdepth}{2} 
\tableofcontents

\vspace{-3\baselineskip}

\section{Introduction}\label{introSEC}

Let $m,n$ be two positive integers and let $d=m+n$. The classical generalized Dirichlet's theorem states that: 
\begin{thm}
For any $A\in \mathrm{M}_{m,n}(\R)$ and $t>1$, there exists $(\bm{p},\bm{q})\in \Z^m\times (\Z^n\setminus\{\bm{0}\})$ satisfying the following system of inequalities:
 \begin{align}\label{equ:dirichletoriginal}
\|A\bm{q}-\bm{p}\|^m\leq \frac{1}{t}\quad \textrm{and}\quad \|\bm{q}\|^n<t.
\end{align}
Here $\|\cdot\|$ denotes the supremum norm on both $\R^m$ and $\R^n$. %
\end{thm}

In view of Dirichlet's theorem, a natural question to ask is whether one can improve \eqref{equ:dirichletoriginal} by replacing $1/t$ by a smaller function, that is, consider the following system of inequalities:
\begin{align}\label{equ:dirichlet}
\|A\bm{q}-\bm{p}\|^m<\psi(t)\quad \textrm{and}\quad \|\bm{q}\|^n<t
\end{align}
where $\psi$ is a positive, continuous, decreasing function which decays to zero at infinity. In this paper, we continue the study of \cite{KleinbockStrombergssonYu2022} on the \textit{metric} aspects of the \textit{uniform approximation} problem on  improvements to
Dirichlet's theorem, %
which aims to understand the Lebesgue measure of the set of matrices $A\in \mathrm{M}_{m, n}(\R)$ for which the  system of inequalities \eqref{equ:dirichlet} is solvable in $(\bm{p},\bm{q})\in \Z^m\times (\Z^n\setminus \{\bm{0}\})$ for \textit{all sufficiently large $t$}.  
We refer the reader to the introduction of \cite{KleinbockStrombergssonYu2022}
for a more extensive review of the related literature. %

As in \cite{KleinbockStrombergssonYu2022}, we consider the slightly more general \textit{weighted} problem in which the supremum norms in \eqref{equ:dirichlet} are replaced by certain weighted quasi-norms.
Let $\bm{\alpha}\in \R^m$ and $\bm{\beta}\in \R^n$ be two \textit{weight vectors}, that is 
\begin{align*}
\bm{\alpha}=(\alpha_1,\ldots, \alpha_m)\in (\R_{>0})^m\quad\textrm{and}\quad \bm{\beta}=(\beta_1,\ldots, \beta_n)\in (\R_{>0})^n
\end{align*}
with $\sum_i\alpha_i=\sum_j\beta_j=1$. Following  \cite{KleinbockStrombergssonYu2022}, we say that $A\in \mathrm{M}_{m,n}(\R)$ is \textit{$\psi_{\bm{\alpha},\bm{\beta}}$-Dirichlet} if the system of inequalities 
\begin{align}\label{equ:dirichletwei}
\|A\bm{q}-\bm{p}\|_{\bm{\alpha}}<\psi(t)\quad \textrm{and}\quad \|\bm{q}\|_{\bm{\beta}}<t
\end{align}
has solutions in $(\bm{p},\bm{q})\in \Z^m\times (\Z^n\setminus \{\bm{0}\})$ \textit{for all sufficiently large $t$}. Here %
\begin{align*}
\|\bm{x}\|_{\bm{\alpha}}:=\max\left\{|x_i|^{1/\alpha_i}\col 1\leq i\leq m\right\}\quad\textrm{and}\quad \|\bm{y}\|_{\bm{\beta}}:=\max\left\{|y_j|^{1/\beta_j}\col 1\leq j\leq n\right\}
\end{align*}
are the two quasi-norms associated with $\bm{\alpha}$ and $\bm{\beta}$ respectively. One easily sees that $A\in \mathrm{M}_{m,n}(\R)$ is $\psi_{\bm{\alpha},\bm{\beta}}$-Dirichlet if and only if $A+A'$ is $\psi_{\bm{\alpha},\bm{\beta}}$-Dirichlet for any $A'\in \mathrm{M}_{m,n}(\Z)$. Thus with slight abuse of notation,  
we may denote by $\DI_{\bm{\alpha},\bm{\beta}}(\psi)\subset \mathrm{M}_{m,n}(\R/\Z)$ the set of $\psi_{\bm{\alpha},\bm{\beta}}$-Dirichlet matrices. One is naturally interested in determining the Lebesgue measure of $\DI_{\bm{\alpha},\bm{\beta}}(\psi)$ for different approximating functions. A particular question of interest is whether $\mathrm{Leb}(\DI_{\bm{\alpha},\bm{\beta}}(\psi))$ satisfies a \textit{zero-one law} depending on
some criterion on the approximating function $\psi$.

In the case $m=n=1$, %
such a zero-one law was established by Kleinbock and Wadleigh \cite[Theorem 1.8]{KleinbockWadleigh2018} using  continued fractions. However, this approach is not applicable for higher dimensions. 
For general $m$ and $n$, this question was studied in \cite{KleinbockStrombergssonYu2022} using homogeneous dynamics via the classical Dani correspondence. In particular, the main result of \cite{KleinbockStrombergssonYu2022} is the following 
 partial zero-one law for $\mathrm{Leb}(\DI_{\bm{\alpha},\bm{\beta}}(\psi))$: %

\begin{thm}[\text{\cite[Theorem 1.2]{KleinbockStrombergssonYu2022}}]\label{thm:improvedirichleconweiold}
Fix $m,n\in \N$ and two weight vectors $\bm{\alpha}\in \R^m$ and $\bm{\beta}\in \R^n$. Let $d=m+n$, and let 
\begin{align}\label{equ:tpars}
\varkappa_d=\frac{d^2+d-4}{2}\quad \textrm{and}\quad \lambda_d=\frac{d(d-1)}{2}.
\end{align}
Let $t_0>0$ and let $\psi: [t_0,\infty)\to {(0,\infty)}$ be a continuous, decreasing function such that
\begin{align}\label{equ:con1psi}
\textrm{the function $t\mapsto t\psi(t)$ is increasing}
\end{align} 
and
\begin{align}\label{equ:con2psi}
\psi(t)<1/t\quad \textrm{for all $t\geq t_0$}.
\end{align} 
Let $F_{\psi}(t):=1-t\psi(t)$. If the series 
\begin{align}\label{equ:crticalseriesconj}
\sum_{k\geq t_0}k^{-1}F_{\psi}(k)^{\varkappa_d}\log^{\lambda_d}\left(\tfrac{1}{F_{\psi}(k)}\right)
\end{align} 
converges, then $\DI_{\bm{\alpha},\bm{\beta}}(\psi)$ is of full Lebesgue measure. Conversely, if the series \eqref{equ:crticalseriesconj} diverges, and %
\begin{align}\label{equ:condipsi}
\liminf_{t_1\to\infty}\frac{\sum_{t_0\leq k\leq t_1}k^{-1}F_{\psi}(k)^{\varkappa_d}\log^{\lambda_d+1}\left(\frac{1}{F_{\psi}(k)}\right)}
{\left(\sum_{t_0\leq k\leq t_1}k^{-1}F_{\psi}(k)^{\varkappa_d}\log^{\lambda_d}\left(\frac{1}{F_{\psi}(k)}\right)\right)^2}=0,
\end{align} 
then $\DI_{\bm{\alpha},\bm{\beta}}(\psi)$ is of zero Lebesgue measure. %
\end{thm}
Note that in the sums in \eqref{equ:crticalseriesconj} and \eqref{equ:condipsi}, $k$ runs through all \textit{integers} subject to the stated condition. This convention will be in force throughout the paper.

\medskip

The following theorem is the main result of the present paper. It
extends Kleinbock and Wadleigh's zero-one law  to general dimensions by removing the assumption \eqref{equ:condipsi} in Theorem \ref{thm:improvedirichleconweiold}. 
It also weakens the monotonicity assumption \eqref{equ:con1psi} by replacing it 
with an assumption %
requiring $F_\psi(t)=1-t\psi(t)$ to be ``quasi-decreasing" in a certain sense (see condition \eqref{equ:con1psinew} below).
\begin{thm}\label{thm:improvedirichleconwei}
Let $d=m+n$, $(\bm\alpha,\bm\beta)\in \R^m\times \R^n$ and $\varkappa_d, \lambda_d$ be as in Theorem \ref{thm:improvedirichleconweiold}. 
Let $t_0>1$ 
and let $\psi: [t_0,\infty)\to {(0,\infty)}$ be a continuous, decreasing function satisfying \eqref{equ:con2psi} and 
\begin{align}\label{equ:con1psinew}
\exists\, C_{\psi}\geq 1\ \text{such that}\ F_{\psi}(t_2)\leq C_{\psi}F_{\psi}(t_1),\qquad \forall\, t_0\leq 
t_1\leq t_2\leq 2t_1.
\end{align} 
Here $F_{\psi}(t)=1-t\psi(t)$ is as before. %
Then we have
\begin{align}\label{equ:zeroonelaw}
\mathrm{Leb}(\DI_{\bm{\alpha},\bm{\beta}}(\psi))=\begin{cases}
1 &\text{if }\: {\displaystyle \sum_{k\geq t_0}k^{-1}F_{\psi}(k)^{\varkappa_d}\log^{\lambda_d}\left(1+\tfrac{1}{F_{\psi}(k)}\right)<\infty},
\\[15pt]
0 &\text{if }\: {\displaystyle \sum_{k\geq t_0}k^{-1}F_{\psi}(k)^{\varkappa_d}\log^{\lambda_d}\left(1+\tfrac{1}{F_{\psi}(k)}\right)=\infty}.
\end{cases}
\end{align}
\end{thm}

\begin{remark}
The particular choice of the factor $2$ in the quasi-decreasing condition \eqref{equ:con1psinew} is clearly inessential; a simple iteration argument shows that it may be replaced by any fixed constant $c>1$.
    
We note that this condition  is, in a sense, essentially the weakest natural regularity assumption needed for the 
dynamical argument used in the present paper to work.
As mentioned earlier, we will use the Dani correspondence to reformulate the problem as a shrinking target problem in terms of a certain function $r$ uniquely determined by $\psi$. The condition \eqref{equ:con1psinew} is then equivalent to an analogous quasi-decreasing condition on $r$ on bounded additive time intervals; see Proposition \ref{prop:dyinput}. 
    
This regularity condition is then used crucially at two points of the argument: First, it allows us to relate the divergence of the series in \eqref{equ:zeroonelaw} to the divergence of a corresponding series involving $r$, which in turn is directly comparable %
to the sum of the measures of the shrinking targets; see \eqref{equ:diequiv}. Second, it is used to pass from the continuous-time shrinking-target problem arising from the Dani correspondence to its discrete-time counterpart; see Lemma \ref{lem:dyred}.
    However, it is not clear to us whether one can, by other methods, substantially weaken this local regularity assumption and still obtain the same zero--one law. 
\end{remark}
\begin{remark}\label{equvserdivREM}
Let $\psi$ be as in Theorem \ref{thm:improvedirichleconwei}. Note that assumption  \eqref{equ:con1psinew} implies that for all $j\gg 1$,
\begin{align}\label{equ:quindi}
C_{\psi}^{-1}F_{\psi}(2^{j+1})\leq F_{\psi}(k)\leq \min\{1, C_{\psi}F_{\psi}(2^j)\},\qquad  \forall\, 2^j\leq k\leq 2^{j+1}.
\end{align}
From this and using the fact that the function $x\mapsto x^{\varkappa_d}\log^{\lambda_d}(1+\tfrac{1}{x})$ is strictly increasing on $(0, 1]$ (which follows from the fact that $\varkappa_d\geq \lambda_d$) we have 
\begin{align}\label{equ:equvserdiv}
\sum_{k}k^{-1}F_{\psi}(k)^{\varkappa_d}\log^{\lambda_d}\left(1+\tfrac{1}{F_{\psi}(k)}\right)=\infty\quad \Leftrightarrow\quad \sum_{j}F_{\psi}(2^j)^{\varkappa_d}\log^{\lambda_d}\left(1+\tfrac{1}{F_{\psi}(2^j)}\right)=\infty.
\end{align}
\end{remark}

\begin{remark}\label{newthmimploldthmREM}
We note that Theorem \ref{thm:improvedirichleconwei} implies Theorem \ref{thm:improvedirichleconweiold}. Indeed, let $\psi$ be as in Theorem \ref{thm:improvedirichleconweiold}, that is,  it satisfies \eqref{equ:con1psi} and \eqref{equ:con2psi}. Then  $0<F_{\psi}(t)<1$ for all $t\geq t_0$ (by \eqref{equ:con2psi}) and $F_{\psi}(t)$ is decreasing in $t$ (by \eqref{equ:con1psi}). This shows that for such $\psi$,
\begin{align}\label{equ:equser}
\sum_{k\geq t_0}k^{-1}F_{\psi}(k)^{\varkappa_d}\log^{\lambda_d}\left(\tfrac{1}{F_{\psi}(k)}\right)=\infty\quad \Leftrightarrow\quad \sum_{k\geq t_0}k^{-1}F_{\psi}(k)^{\varkappa_d}\log^{\lambda_d}\left(1+\tfrac{1}{F_{\psi}(k)}\right)=\infty.
\end{align} 
On the other hand, note that condition \eqref{equ:con1psi} is a special case of \eqref{equ:con1psinew} with $C_{\psi}=1$. Hence  \eqref{equ:zeroonelaw} is applicable for such $\psi$, which, together with \eqref{equ:equser},   clearly  implies the conclusion of Theorem \ref{thm:improvedirichleconweiold}.

We mention that the reason we change the series in \eqref{equ:crticalseriesconj} to the series in \eqref{equ:zeroonelaw} is that for $\psi$ satisfying the assumptions in Theorem \ref{thm:improvedirichleconwei}, it is possible that $\liminf_{t\to\infty}t\psi(t)=0$, or equivalently, $\limsup_{t\to\infty}F_{\psi}(t)=1$. When this happens, \eqref{equ:quindi} gives $\limsup_{j\to\infty}F_\psi(2^j)>0$. By \eqref{equ:equvserdiv}, we are therefore in the divergence case of \eqref{equ:zeroonelaw}, and  we will show that  $\mathrm{Leb}(\DI_{\bm{\alpha},\bm{\beta}}(\psi))=0$. %
However, since $\lim_{t\to 1^+}\log(t)=0$,
a priori it is possible that the series in \eqref{equ:crticalseriesconj}  converges in this case.
\end{remark}

\subsection{The short-range correlation argument}
Let us briefly explain the main new ingredient in the proof of the divergence part of Theorem \ref{thm:improvedirichleconwei}. After applying the Dani correspondence and passing to discrete time, one needs to control the correlations between hitting events at two different times. The effective double equidistribution
estimates recalled in Section \ref{equidoublemixingSEC} give the required quasi-independence when the two times are sufficiently far apart. However, as explained in \cite[Remark 8]{KleinbockStrombergssonYu2022},  the characteristic functions of the shrinking targets have to be approximated by smooth functions whose norms grow as the targets shrink. Because of this smoothing issue, the effective estimates are not strong enough when the two times are close. In \cite{KleinbockStrombergssonYu2022} these short-range terms were treated using a trivial estimate, which introduced an additional logarithmic loss and led to the extra assumption \eqref{equ:condipsi}.

To overcome this difficulty, we treat these short-range correlations in a different way. Instead of using the original targets, we use a family of smaller sets defined in \eqref{def:tildedeltar'}. On the one hand, these sets have the same order of measure as the original targets; see Theorem \ref{thm:meaestss} and Proposition \ref{prop:meaest}. On the other hand,  they satisfy an additional geometric disjointness property. Roughly speaking, suitable translates of these sets are pairwise disjoint over a  time interval of logarithmic length; see Proposition \ref{prop:shortmix}. This disjointness fact gives a much better correlation estimate
for hitting events at nearby times. 
For pairs of times which are farther apart, we can use the effective double equidistribution  estimates as before. Combining these two types of correlation estimates allows us to remove condition \eqref{equ:condipsi}  and prove Theorem \ref{thm:improvedirichleconwei}. 

In recent related work, Pfitscher, Rao, Yu and Zhang \cite{PfitscherRaoYuZhang2026} adapt the same short-range strategy to uniform approximation over the Gaussian and Eisenstein integers. The geometry of the critical locus, and hence the construction of the appropriate shrinking targets, is different in the two complex cases. Nevertheless, the needed disjointness for nearby times is proved in \cite[Proposition~5.4]{PfitscherRaoYuZhang2026}. This is used to control the short-range correlations, while effective double equidistribution is used when the two times are farther apart. This application indicates that the selected-target and disjointness mechanism is not restricted to the real setting considered here.

\subsection{Notation and conventions}\label{notconvSEC}

A real-valued function $f$ defined on an interval $I\subset\R$ is called \textit{increasing} (resp.\ \textit{decreasing}) if $f(t_1) \leq f(t_2)$ (resp.\ $f(t_1) \geq f(t_2)$) whenever $t_1 < t_2$.
For two positive quantities $A$ and $B$, we write ``$A\ll B$" to denote $A\leq cB$ for some constant $c>0$, and we write
``$A\asymp B$" to denote $A\ll B\ll A$. We will also write ``$O(A)$" to denote any real number $E$ satisfying $|E|\leq cA$ for some constant $c>0$. We will sometimes use subscripts to indicate the dependence of the implicit constant (i.e., the constant $c$ above) on parameters.

\vspace{5pt}

\noindent
\textbf{AI disclosure:}
The first version (v1) of this work was uploaded to arXiv in February 2026. Since August 2026, the authors have used ChatGPT (OpenAI, GPT-5.6 Sol) for mathematical discussion and proofreading. In particular, ChatGPT assisted in developing the argument establishing Theorem \ref{thm:improvedirichleconwei} under the regularity condition \eqref{equ:con1psinew}. In the February version, the theorem was proved under the same condition on the larger range
$t_0\leq t_1\leq t_2\leq t_1e^{(\log t_1)^{\eta}}$
($\eta>0$ fixed),
which the authors had subsequently reduced to
$t_0\leq t_1\leq t_2\leq t_1\log t_1$
before using AI assistance. All other parts of the proof were obtained and written by the authors. The AI-assisted argument was revised and rewritten by the authors, who independently verified the complete argument and take full responsibility for the final proof.

\vspace{5pt}

\noindent\textbf{Acknowledgements.}
We would like to thank Ren\'e Pfitscher and Han Zhang for helpful discussions.

\section{Measure estimate of a certain family of sets}\label{measestSEC}
Fix an integer $d\geq 2$. Let $G=\SL_d(\R)$, $\G=\SL_d(\Z)$ and let $X_d=G/\G$ be the space of unimodular lattices in $\R^d$. Let $\pi: G\to X_d$ be the natural projection from $G$ to $X_d$. Let $\mu_d$ be the unique invariant probability measure on $X_d$. Let $\Delta: X_d\to [0,\infty)$ be the function defined by
\begin{align}\label{def:delta}
\Delta(\Lambda):=\sup_{\bm{v}\in \Lambda\setminus \{\bm{0}\}}\log\Bigl(\tfrac{1}{\|\bm{v}\|}\Bigr).
\end{align}
We note that the fact that $\Delta(\Lambda)\geq 0$ for any $\Lambda\in X_d$ is a consequence of Minkowski's convex body theorem, %
while the level set $\Delta^{-1}\{0\}$ is the \textit{critical locus} of $X_d$ 
with respect to the supremum norm, which, by Haj\'os's Theorem \cite{Hajos1941}, is a finite union of explicit compact submanifolds of codimension
$\frac{(d-1)(d+2)}2$, 
cf. \cite[Theorem 2.3]{KleinbockStrombergssonYu2022}. 
For our purposes, we are naturally interested 
in the level sets $\Delta^{-1}[0,r]$ for small $r>0$ as they will describe the target sets in our shrinking target problem. 
Indeed, the following measure estimate on $\Delta^{-1}[0,r]$ was proved in \cite{KleinbockStrombergssonYu2022}:
\begin{Thm}[\text{\cite[Theorem 1.3]{KleinbockStrombergssonYu2022}}]\label{KSYTHM1p3}
We have
\begin{align}\label{equ:meestksy}
\mu_d\left(\Delta^{-1}[0, r]\right)\asymp_d r^{\frac{(d-1)(d+2)}{2}}\log^{\frac{d(d-1)}{2}}\Bigl(\tfrac{1}{r}\Bigr),%
\qquad \textrm{as $r\to 0^+$}.
\end{align}
\end{Thm}
For our purpose of removing the technical assumption \eqref{equ:condipsi}, we need to study a more carefully chosen subset of $\Delta^{-1}[0,r]$ which we now introduce. For any Borel set $\mathcal{A}\subset \R^d$, set
$$
\cH(\mathcal{A}):=\left\{\Lambda\in X_d: (\Lambda\setminus\{\bm{0}\})\cap \mathcal{A}\neq \emptyset\right\}.
$$
For any $0<r<1$, 
let 
$$
R_r:=(1-\tfrac{r}{2d}, 1+\tfrac{r}{2d})\times (-\sqrt{r}, \sqrt{r})^{d-1}\subset \R^d,
$$ 
and define
\begin{align}\label{equ:deltar'}
\Delta'_r:=\Delta^{-1}[0, r]\cap \cH(R_r),\qquad \forall\, 0<r<1. %
\end{align}
We note that  compared to $\Delta^{-1}[0,r]$,
$\Delta_r'$ satisfies the following additional property: 
\begin{align}\label{equ:keyproperty}
    \Lambda\in \Delta_r'\quad \Rightarrow\quad \exists\, \bm{v}\in \Lambda\ \text{such that}\ \bm{v}\in R_r=(1-\tfrac{r}{2d}, 1+\tfrac{r}{2d})\times (-\sqrt{r}, \sqrt{r})^{d-1}. 
\end{align}
This property will allow us to prove a disjointness result over a larger time range than would be available if $\Delta^{-1}[0,r]$ were used directly as the target set, as in \cite{KleinbockStrombergssonYu2022};
see Proposition \ref{prop:shortmix} below. 

On the other hand, we show below that the $\mu_d$-measure of $\Delta_r'$ is comparable to that of $\Delta^{-1}[0,r]$:
\begin{Thm}\label{thm:meaestss}
We have
\begin{align}\label{equ:meaestss}
\mu_d\left(\Delta'_r\right)\asymp_d r^{\frac{(d-1)(d+2)}{2}}\log^{\frac{d(d-1)}{2}}\Bigl(\tfrac{1}{r}\Bigr),%
\qquad \textrm{as $r\to 0^+$}.
\end{align}
\end{Thm}

\subsection{Recap of the lower bound in (\ref*{equ:meestksy})}\label{lowerboundrecapSEC}
Since $\Delta'_r\subset \Delta^{-1}[0,r]$, in view of \eqref{equ:meestksy}, in order to prove Theorem \ref{thm:meaestss}, we only need to prove the lower bound in \eqref{equ:meaestss}. We will follow closely the approach in \cite{KleinbockStrombergssonYu2022}. Let us first review the results from \cite{KleinbockStrombergssonYu2022} used to prove the lower bound in \eqref{equ:meestksy}. %

We first introduce the following  coordinates; %
see \cite[Section 2.1]{KleinbockStrombergssonYu2022} for more details.  
Let $\{\bm{e}_i: 1\leq i\leq d\}$ be the standard orthonormal basis of $\R^d$. 
Let 
\begin{align*}
P=\{p\in G\col 
p\bm{e}_d=t\bm{e}_d\ \textrm{for some $t\neq 0$}\}<G
\end{align*}
be the maximal parabolic subgroup fixing the line spanned by $\bm{e}_d\in \R^d$, and let 
\begin{align*}
N=\left\{u_{\bm{x}}:=\left(\begin{smallmatrix}
I_{d-1} & \bm{x}\\
\bm{0}^t & 1\end{smallmatrix}\right)\col \bm{x}\in \R^{d-1}\right\}<G
\end{align*}
 be the transpose of the unipotent radical of $P$.
 
Let $\nu$ be the Haar measure of $G$ that locally agrees with $\mu_d$. After removing a null set, every element $g\in G$ can be written uniquely as a product:
\begin{align}\label{puCOORD}
g=p_{\bm{b}_1,\ldots, \bm{b}_{d-1}}u_{\bm{x}}\quad{\text{for some $p_{\bm{b}_1,\ldots, \bm{b}_{d-1}}\in P$ and }u_{\bm{x}}\in N}.
\end{align}
Here $p=p_{\bm{b}_1,\cdots,\bm{b}_{d-1}}\in P$ is the unique element of $ P$ satisfying  $p\bm{e}_i=\bm{b}_i$ for all $1\leq i\leq d-1$. 
In terms of the coordinates in \eqref{puCOORD}, %
$\nu$ is given by 
\begin{equation}\label{equ:haar}
\mathrm{d}\nu(g)=\frac{1}{\zeta(2)\cdots \zeta(d)}\, \mathrm{d} \bm{x}\prod_{1\leq i\leq d-1}\mathrm{d}\bm{b}_i,
\end{equation}
where $\zeta(\cdot)$ is the Riemann zeta function,
and $\mathrm{d}\bm{x}$ and $\mathrm{d}\bm{b}_i$ denote Lebesgue measure on $\R^{d-1}$ and $\R^d$,
respectively.

For any $r\in (0,1)$, let 
\begin{align*}
K_r:=\left\{\Lambda\in X_d: \Lambda\cap (r-1,1-r)^d=\{\bm{0}\}\right\}.
\end{align*}
Note that $\Delta^{-1}[0,r]=K_{1-e^{-r}}$ (or equivalently, $K_r=\Delta^{-1}[0,-\log(1-r)]$) and $1-e^{-r}=r+O(r^2)\asymp r$; hence to show the lower bound in \eqref{equ:meestksy}, it suffices to show that 
\begin{align}\label{equ:krmeaest}
\mu_d(K_r)\gg_d r^{\frac{(d-1)(d+2)}{2}}\log^{\frac{d(d-1)}{2}}\Bigl(\tfrac{1}{r}\Bigr),%
\qquad \textrm{as $r\to 0^+$}.
\end{align}
The following proposition, proved in \cite{KleinbockStrombergssonYu2022}, implies the above lower bound; see \cite[Proposition 3.2]{KleinbockStrombergssonYu2022} and the computation after it, \cite[pp.\ 796-797]{KleinbockStrombergssonYu2022}. 

\begin{Prop}\label{prop:regulargen}
There exists a small parameter $c_0\in (0,1)$ depending only on $d$ such that the following holds: For any  $r\in (0,c_0/d)$, define $\underline{\cK}_r\subset G$ to consist of those elements $g=p_{\bm{b}_1,\dots,\bm{b}_{d-1}}u_{\bm{x}}$ for which  
$\bm{b}_j=(b_{1j},\ldots,b_{dj})^t\in\R^d$ ($j=1,\ldots,d-1$)
and $\bm{x}\in (0,c_0/d)^{d-1}$ satisfy the following conditions:
\begin{align}\label{equ:con1}
b_{ij}, -b_{ji}\in (-c_0, 0),\ \forall\ 1\leq j< i\leq d-1; \ b_{d\ell}\in (-c_0, 0), b_{\ell \ell}\in (1-\tfrac{r}{2d},1), \  \forall\ 1\leq \ell\leq d-1,%
\end{align}
\begin{align}\label{equ:con2}
b_{ij}<db_{i,j-1}\ (\Leftrightarrow |b_{ij}|>d|b_{i,j-1}|),\qquad \forall\ 2\leq j<i\leq d,
\end{align}
\begin{align}\label{equ:con3}
|b_{ij}b_{ji}|<\frac{r}{d!},\quad \forall 1\leq j<i\leq d-1,\qquad \textrm{and}\qquad \sum_{j=1}^{d-1}|b_{dj}|x_j<\frac{r}{2},
\end{align}
and
\begin{align}\label{equ:con4}
b_{kj}>b_{ij}\ (\Leftrightarrow |b_{kj}|<|b_{ij}|),\qquad \forall 1\leq j<k<i\leq d.
\end{align}
Then $\pi|_{\underline{K}_r}: \underline{\cK}_r\to X_d$ is injective, $\underline{K}_r:=\pi(\underline{\cK}_r)\subset K_r$, and 
\begin{align*}
\mu_d(\underline{K}_r)\gg_d r^{\frac{(d-1)(d+2)}{2}}\log^{\frac{d(d-1)}{2}}\Bigl(\tfrac{1}{r}\Bigr),%
\qquad \textrm{as $r\to 0^+$}.
\end{align*}
\end{Prop}
\subsection{Proof of Theorem \ref*{thm:meaestss}}
We now give the proof of Theorem \ref{thm:meaestss}. As we have discussed above, it suffices to prove the lower bound in  \eqref{equ:meaestss}. 

For any $0<r<(c_0/d)^2$, define {$\underline{\cK}'_r$ to be the 
subset of $\underline{\cK}_r$ consisting of all 
$g=p_{\bm{b}_1,\dots,\bm{b}_{d-1}}u_{\bm{x}}\in\underline{\cK}_r$
which satisfy the extra condition
\begin{align}\label{EQU:addcond}
b_{i1}\in (-\sqrt{r}, 0),\,\qquad \forall\, 2\leq i\leq d. 
\end{align}
In other words, an element $g=p_{\bm{b}_1,\dots,\bm{b}_{d-1}}u_{\bm{x}}$ in $G$
with $\bm{b}_j=(b_{1j},\ldots,b_{dj})^t\in\R^d$ ($j=1,\ldots,d-1$)
belongs to $\underline{\cK}'_r$ if and only if
$\bm{x}\in (0,c_0/d)^{d-1}$ and all five conditions
\eqref{equ:con1}, \eqref{equ:con2}, \eqref{equ:con3}, \eqref{equ:con4}
and \eqref{EQU:addcond} hold.
Note that the new 
condition \eqref{EQU:addcond}
is a sharpening of the restrictions
$b_{i1}\in (-c_0, 0)$ ($\forall\, 2\leq i\leq d$)
which are part of condition \eqref{equ:con1}.} %

Set also 
\begin{align}\label{equ:kr'}
\underline{K}_r':=\pi(\underline{\cK}_r')\subset X_d.
\end{align}
Then by Proposition \ref{prop:regulargen} we have 
$$
\underline{K}_r'\subset \underline{K}_r\subset K_r.
$$ 
Moreover, because of  condition \eqref{EQU:addcond}  and $b_{11}\in (1-\frac{r}{2d}, 1)$ (by \eqref{equ:con1}), for any $\Lambda=p_{\bm{b}_1,\dots,\bm{b}_{d-1}}u_{\bm{x}}\mathbb{Z}^d \in \underline{K}_r'$ with $p_{\bm{b}_1,\dots,\bm{b}_{d-1}}u_{\bm{x}}\in \underline{\cK}'_r$, we have 
$$
\bm{b}_1= p_{\bm{b}_1,\dots,\bm{b}_{d-1}}u_{\bm{x}}\bm{e}_1\in (\Lambda\setminus\{\bm{0}\})\cap \left((1-\tfrac{r}{2d},1)\times (-\sqrt{r}, 0)^{d-1} \right)\subset (\Lambda\setminus\{\bm{0}\})\cap R_r.
$$
This implies that $\Lambda\in \cH(R_r)$. 
We have thus proved that $\underline{K}_r'\subset K_r\cap \cH(R_r)$ for every 
$0<r<(c_0/d)^2$; hence
for every such $r$ we also have, using $1-e^{-r}<r$:
\begin{align}\label{equ:inclukeyALT}
\underline{K}_{1-e^{-r}}'\subset K_{1-e^{-r}}\cap \cH(R_{1-e^{-r}})=\Delta^{-1}[0,r]\cap \cH(R_{1-e^{-r}})\subset \Delta^{-1}[0,r]\cap \cH(R_r)=\Delta'_r.
\end{align}
In view of this relation, and the fact that $1-e^{-r}\asymp r$, in order to prove the lower bound in \eqref{equ:meaestss}, it suffices to prove the following measure estimate on $\mu_d(\underline{K}_r')$.

\begin{Prop}\label{prop:meaestlobd}
We have
\begin{align}\label{equ:kr'meest}
\mu_d(\underline{K}_r')\gg_d r^{\frac{(d-1)(d+2)}{2}}\log^{\frac{d(d-1)}{2}}\Bigl(\tfrac{1}{r}\Bigr),%
\qquad \textrm{as $r\to 0^+$}.
\end{align}
\end{Prop}
\begin{proof}
Fix $0<r<c^2_0/d^{2d-2}$. Since $\underline{\cK}_r'\subset \cK_r$ and $\pi|_{\underline{\cK}_r}$ is injective (by Proposition \ref{prop:regulargen}),  $\pi|_{\underline{\cK}_r'}$ is also injective. This implies that 
\begin{align*}
\mu_d(\underline{K}_r')=\nu(\underline{\cK}_r'). 
\end{align*}
It now remains to compute $\nu(\underline{\cK}_r')$. This computation is very similar to the computation of $\nu(\underline{\cK}_r)$ done in \cite[pp.\ 796-797]{KleinbockStrombergssonYu2022}; we follow closely the strategy there. 

First, by \eqref{equ:haar} we have 
\begin{align*}
\nu\left({\underline\cK'_r}\right)&\asymp_d \prod_{1\leq k\leq d-1}\left(\int_{1-\frac{r}{2d}}^1\mathrm{d} b_{kk}\right)
\int_{\cR}\delta\bigl(\,(b_{ij})_{1\leq j<i\leq d}\,\bigr) \prod_{1\leq j<i\leq d}\mathrm{d} b_{ij},
\end{align*}
where
$$
\cR:=\left\{(b_{ij})_{1\leq j<i\leq d}\in (-c_0, 0)^{d(d-1)/2} \col \textrm{$(b_{ij})$ satisfies \eqref{equ:con2}, \eqref{equ:con4} and \eqref{EQU:addcond}} \right\},
$$
and
\begin{align*}
\delta\bigl(\,(b_{ij})_{1\leq j<i\leq d}\,\bigr)
:&=\prod_{1\leq j<i\leq d-1}\int_0^{\min\left\{c_0,\frac{r}{d!|b_{ij}|}\right\}}\mathrm{d} b_{ji}\times \int_{\left\{\bm{x}\in \left(0,\tfrac{c_0}{d}\right)^{d-1}\col \sum_{j=1}^{d-1}|b_{dj}|x_j<\frac{r}{2}\right\}}\prod_{1\leq j\leq d-1}\mathrm{d} x_j\\
&\asymp_{d}\prod_{1\leq j<i\leq d}\min\left\{1,\tfrac{r}{|b_{ij}|}\right\}.
\end{align*}
Hence 
\begin{align}\label{muduKrlb}
\mu_d\left(\underline{K}'_r\right)=\nu\left({\underline\cK'_r}\right)\asymp_{d}r^{d-1}\int_{\cR}\left(\prod_{1\leq j<i\leq d} \min\left\{1, \tfrac{r}{|b_{ij}|}\right\}\mathrm{d} b_{ij}\right).
\end{align}

Now we make the change of variables $b_{ij}=-d^{j-1}z_{ij}$, $1\leq j<i\leq d$,
so that all the variables $z_{ij}$ are positive, and the ordering conditions \eqref{equ:con2} and \eqref{equ:con4} become
\begin{align}\label{equ:ordering}
z_{i'j'}<z_{ij} \quad \textrm{whenever 
$1\leq j'<i'\leq d$, $\:1\leq j<i\leq d$,
$\:i'\leq i$, $\:j'\leq j$ and $(i',j')\neq (i, j)$}.
\end{align}
Moreover, for any $2\leq j<i$, the condition $b_{ij}\in (-c_0,0)$ corresponds to $z_{ij}\in (0,c_0/d^{j-1})$,
and we note that each of these intervals contains the smaller interval
$(0,c_0/d^{d-1})$. For each $2\leq i\leq d$,  the condition $b_{i1}\in (-\sqrt{r}, 0)$ becomes $z_{i1}\in (0, \sqrt{r})$ which contains the smaller interval $(r, \sqrt{r})$. Note that since  $r<c^2_0/d^{2d-2}$, we have $(r, \sqrt{r})\subset  (0,c_0/d^{d-1})$. 

{The key modification of the computation in \cite{KleinbockStrombergssonYu2022} is now that we restrict  each $z_{ij}$ to the smaller interval $(r,\sqrt{r})$.}
Then $r/|b_{ij}|=r d^{1-j}/z_{ij}<d^{1-j}\leq 1$,
so that 
$$
\min\left\{1,r/|b_{ij}|\right\}=rd^{1-j}/z_{ij}\asymp_d r/z_{ij},\quad \forall\, 1\leq j<i\leq d.
$$
It now follows from the same arguments as in \cite[p.\ 797]{KleinbockStrombergssonYu2022}  that 
\begin{align}
\mu_d\left(\underline{K}'_r\right)&\gg_{d, c_0} r^{\frac{(d-1)(d+2)}{2}}\prod_{1\leq j<i\leq d}\left(\int_r^{\sqrt{r}}\frac{\mathrm{d} z_{ij}}{z_{ij}}\right)
\asymp r^{\frac{(d-1)(d+2)}{2}}\log^{\frac{d(d-1)}{2}}\Bigl(\tfrac{1}{r}\Bigr).
\end{align}
This finishes the proof of the proposition.
\end{proof}

\section{Some preparations for the proof}
\label{proofprepSEC}

\subsection{Single and double equidistribution of certain \texorpdfstring{$g_s$}{gs}-translates}
\label{equidoublemixingSEC}
Let $m,n\in \N$ be two positive integers and let $d=m+n$ as before. 
Let us denote
\begin{align}\label{equ:submainfold}
\mathcal{Y}:=\left\{\Lambda_A:= \left(\begin{matrix}
I_m & A\\
0 & I_n\end{matrix}\right)\Z^d\in X_d\col  A\in \mathrm{M}_{m,n}(\R)\right\}.
\end{align}
The submanifold $\mathcal{Y}\subset X_d$ can be naturally identified with the $mn$-dimensional torus $\mathrm{M}_{m,n}(\R/\Z)$ via $\Lambda_A\leftrightarrow A\in \mathrm{M}_{m,n}(\R/\Z)$. Let $\textrm{Leb}$ be the probability Lebesgue measure on $\mathcal{Y}\cong \mathrm{M}_{m,n}(\R/\Z)$;  for any function $f$ on $\scrY$ we denote the space average of $f$ on $\cY$ by %
$\int_{\scrY}f(\Lambda_A)\,\mathrm{d} A$.

Let $\bm{\alpha}$ and $\bm{\beta}$ be the two weight vectors  as in Theorem \ref{thm:improvedirichleconwei}. Let
\begin{align}\label{equ:geneweflow}
g_s=g_{s}^{\bm{\alpha},\bm{\beta}}:=\left(\begin{smallmatrix} 
e^{\alpha_1s} & &  & & &\\
& \ddots & & & & \\
& & e^{\alpha_ms} & & &\\
& & & e^{-\beta_1s} & &\\
 & & & & \ddots & \\
 & & & & & e^{-\beta_ns}\end{smallmatrix}\right)\in G\qquad (s\in \R),
\end{align}
be the one-parameter diagonal subgroup associated to $\bm\alpha$ and $\bm\beta$.  The most important dynamical input for our argument is the following effective single and double equidistribution theorems for the $g_s$-translates of $\scrY$, respectively proved by Kleinbock-Margulis \cite[Theorem 1.3]{KleinbockMargulis2012}
and Kleinbock-Shi-Weiss \cite[Theorem 1.2]{KleinbockShiWeiss2017}. 
{Our main reference will be \cite[Sec.\ 2]{BjorklundGorodnik2019}, where error bounds are given with explicit dependence on the test functions.}

First let us introduce the norm which we will use in the statement of the theorem.
Denote by $C_c^{\infty}(X_d)$ the space of compactly supported smooth functions on $X_d$.
Let $\mathfrak{g}=\mathfrak{sl}_d(\R)$ be the Lie algebra of $G$. Each element $Y\in \mathfrak{g}$ acts on $C_c^{\infty}(X_d)$ as a first-order differential operator  via the Lie derivative formula, which we denote by $\cD_Y$. %
Fix an ordered basis $\{Y_1,\ldots, Y_{a}\}$ of $\mathfrak{g}$. 
Then every monomial $Z=Y_1^{\ell_1}\cdots Y_{a}^{\ell_{a}}$ defines a 
differential operator of degree $\deg(Z):=\ell_1+\cdots+\ell_{a}$ via
\begin{align*}
\mathcal{D}_Z:=\mathcal{D}_{Y_1}^{\ell_1}\cdots\mathcal{D}_{Y_{a}}^{\ell_{a}}.
\end{align*}
Now for any $\ell\in\N$, we define the \textit{$C^{\ell}$-norm} on $C_c^{\infty}(X_d)$ through
\begin{align*}
\|f\|_{C^{\ell}}:=\sum_{\operatorname{deg}(Z)\leq\ell}\sup_{x\in X_d}\bigl|(\mathcal{D}_Zf)(x)\bigr| \qquad (f\in C_c^{\infty}(X_d)),
\end{align*}
where the summation is over all the monomials $Z$ in $\{Y_1,\ldots, Y_a\}$ with degree at most $\ell$. 

The following result is a consequence of 
the proof of
\cite[Theorem 2.2]{BjorklundGorodnik2019}.

\begin{Thm}\label{thm:effequhorsp}
There exist $\ell\in\N$ and $\delta>0$ such that for any $f\in C^{\infty}_c(X_d)$ and $s>0$,
\begin{align}\label{equ:effequ}
\int_{\mathcal{Y}}f(g_s\Lambda_A)\,\mathrm{d} A=\mu_d(f)+O\left(\|f\|_{C^{\ell}}\, e^{-\delta s}\right),
\end{align}
and for any $f_1,f_2\in C_c^{\infty}(X_d)$ and $s_2>s_1>0$, 
\begin{align}\label{equ:dmix2}
\int_{\mathcal{Y}}f_1(g_{s_1}\Lambda_A)f_2(g_{s_2}\Lambda_A)\,\text{d} A=\mu_d(f_1)\mu_d(f_2)+O\left(\|f_1\|_{C^{\ell}}\|f_2\|_{C^{\ell}}\,e^{-\delta(s_2-s_1)}
+\|f_1\|_{C^{\ell}}\,\bigl|\mu_d(f_2)\bigr|\,e^{-\delta s_1}\right).
\end{align}
\end{Thm}
\begin{proof}[Proof of Theorem \ref{thm:effequhorsp}]
First note that \eqref{equ:effequ} is a special case of  \cite[Cor.\ 2.4]{BjorklundGorodnik2019} by taking ``$\phi_0=1$" and ``$r=1$" there and noting that the $C^\ell$-norm is stronger than the norm ``$\mathcal{N}_{\ell}$''
used in the statement of \cite[Cor.\ 2.4]{BjorklundGorodnik2019}.

For \eqref{equ:dmix2}, we use the following more general double equidistribution estimate, which follows from the proof of \cite[Eq.\ (2.26)]{BjorklundGorodnik2019} and the discussion following it: 
Up to reducing $\delta>0$ and enlarging $\ell\in \N$ in \eqref{equ:effequ}, we have for every $f_0\in C_c^{\infty}(\mathrm{M}_{m,n}(\R))$, 
$f_1, f_2\in C_c^{\infty}(X_d)$, and $s_2>s_1>0$,
    \begin{align}\notag
    \int_{\mathrm{M}_{m,n}(\R)}f_0(A)f_1(g_{s_1}\Lambda_A)f_2(g_{s_2}\Lambda_A)\,\mathrm{d}A
    =\biggl(\int_{\mathrm{M}_{m,n}(\R)}f_0(A)f_1(g_{s_1}\Lambda_A)\,\mathrm{d}A\biggr)\cdot\mu_d(f_2)
    \hspace{40pt}
    \\\label{equ:bgindfor}
    +O_{f_0}\left(e^{-\delta (s_2-s_{1})}\|f_1\|_{C^{\ell}}\|f_2\|_{C^{\ell}}\right).
    \end{align}
To see how to deduce \eqref{equ:bgindfor} from the proof of \cite[Eq.\ (2.26)]{BjorklundGorodnik2019}, let
\begin{align*}
\bm{w}=\bm{w}_{(s_1,s_2)}:=\left(\alpha_1s_2-\tfrac{nb(s_2-s_1)}{2},\cdots, \alpha_ms_2-\tfrac{nb(s_2-s_1)}{2}, \beta_1s_2-\tfrac{mb(s_2-s_1)}{2},\cdots, \beta_ns_2-\tfrac{mb(s_2-s_1)}{2}\right),
\end{align*}
   where $b:=\min\{\frac{\alpha_i}{n}, \frac{\beta_j}{m}: 1\leq i\leq m, 1\leq j\leq n\}$. 
  Let us also introduce the notation
   $\lfloor \bm{a}\rfloor:=\min\{a_i: 1\leq i\leq d\}$
   for any $\bm{a}\in (\R_{>0})^{d}$, and 
   set $\bm{r}_t:=\bigl(\alpha_1t,\ldots,\alpha_mt,\beta_1t,\ldots,\beta_nt\bigr)
   \in \R^d$ %
   for any $t>0$.
   One then immediately verifies that the vector $\bm{w}$ satisfies the following three properties:
   \begin{enumerate}
\item[(i)] $\lfloor\bm{w} \rfloor \geq \frac{b}{2}(s_2+s_1)$, 
\item[(ii)] $\lfloor \bm{w}-\bm{r}_{s_1}\rfloor \geq \frac{b}{2}(s_2-s_1)$,
\item[(iii)] $\bm{r}_{s_2}-\bm{w}=\bigl(\tfrac{z}{m},\cdots, \tfrac{z}{m},\tfrac{z}{n},\cdots, \tfrac{z}{n}\bigr)$ %
with $z=\frac{mnb}{2}(s_2-s_1)$. 
\end{enumerate}
By following the arguments in 
\cite[pp.\ 1384-1388]{BjorklundGorodnik2019} (for the special case $r=2$), but using the above vector $\bm{w}$ in place of ``$\overline{s}\,$", and using the above estimates (i)--(iii)  in place of \cite[(2.15)--(2.17)]{BjorklundGorodnik2019}, we obtain \eqref{equ:bgindfor}.

   Next we mimic the argument in \cite[Remark 12]{edwards21}:
     We fix, once and for all, a choice of a function $f_0\in C_c^{\infty}(\mathrm{M}_{m,n}(\R))$ satisfying %
     $\sum_{N\in \mathrm{M}_{m,n}(\Z)}f_0(A+N)=1$ for all $A\in \mathrm{M}_{m,n}(\R)$. Note that   
     \begin{align*}
         \int_{\mathrm{M}_{m,n}(\R)}f_0(A)f(\Lambda_A)\,\mathrm{d}A=\int_{\cY}f(\Lambda_A)\,\mathrm{d}A,\qquad \forall\, f\in C_c^{\infty}(X_d). 
     \end{align*}
    Now  \eqref{equ:dmix2} follows from \eqref{equ:bgindfor} applied with our fixed choice of $f_0$, and combined with \eqref{equ:effequ}.
\end{proof}

We also record here the following properties satisfied by the $C^{\ell}$-norm:
\begin{Lem}\label{lem:sobonorm}
For any $\ell\in \N$ we have
\begin{enumerate}
\item[(i)] $\|f_1f_2\|_{C^\ell} \ll_{\ell} \|f_1\|_{C^\ell}\|f_2\|_{C^\ell}$ for any $f_1,f_2\in C_c^{\infty}(X_d)$.
\item[(ii)] $\|g\cdot f\|_{C^\ell}\ll_{\ell} \|\mathrm{Ad}(g^{-1})\|_{\rm op}^{\ell}\|f\|_{C^\ell}$, for any $g\in G$ and $f\in C_c^{\infty}(X_d)$, where $(g\cdot f)(x):=f(g^{-1}x)$ for any $x\in X_d$, $\mathrm{Ad}: G\to \operatorname{GL}(\fg)$ is the adjoint representation of $G$ and $\|\mathrm{Ad}(g^{-1})\|_{\rm op}$ is the operator norm of $\mathrm{Ad}(g^{-1})$ with respect to any fixed norm on $\fg$. 
\end{enumerate}
\end{Lem}

\begin{remark}\label{rmk:operanorm}
We mention that the implied constant in the above bound in item (ii) also depends on the choice of norm on $\fg$ and  on the choice of ordered basis $\{Y_1,\ldots,Y_a\}$ of $\fg$. %
Moreover, by decomposing $\fg$ into eigenspaces of $\mathrm{Ad}(g_s)$, we see that

\begin{align*}
\|\mathrm{Ad}(g_s)\|_{\rm op}\ll e^{(\alpha_{\max}+\beta_{\max})|s|}\leq e^{2|s|},\qquad \forall\, s\in\R. 
\end{align*}
Here $\alpha_{\max}:=\max\{\alpha_i: 1\leq i\leq m\}$ and $\beta_{\max}:=\max\{\beta_j: 1\leq j\leq n\}$. 
\end{remark}
\begin{proof}[Proof of Lemma \ref{lem:sobonorm}]
Property (i) is immediate from Leibniz's rule.
For (ii), note that for any monomial $Z=Y^{\ell_1}_1\cdots Y_a^{\ell_a}$ of degree $\ell_Z=\ell_1+\cdots+\ell_a\leq \ell$, %
$$
\cD_{Z}(g\cdot f)=g\cdot (\cD_{\mathrm{Ad}(g^{-1})(Y_1)^{\ell_1}}\circ \cdots \circ\cD_{\mathrm{Ad}(g^{-1})(Y_a)^{\ell_a}}f). 
$$
From this one sees that 
$$
\sup_{x\in X_d}\bigl|\cD_Z(g\cdot f)(x)\bigr|
=\sup_{x\in X_d}\bigl|\bigl(\cD_{\mathrm{Ad}(g^{-1})(Y_1)}^{\ell_1} \cdots 
\cD_{\mathrm{Ad}(g^{-1})(Y_a)}^{\ell_a}f\bigr)(x)\bigr|
\ll_{\ell_Z}\|\mathrm{Ad}(g^{-1})\|_{\rm op}^{\ell_Z}\|f\|_{C^{\ell}}.
$$
This implies the bound in (ii). 
\end{proof}

\subsection{Reduction to dynamics}

In this section we briefly review the classical \textit{Dani correspondence} observed by Dani \cite{Dani1985} and Kleinbock-Margulis \cite[Lemma 8.3]{KleinbockMargulis1999}, which allows us to reduce the problem to  a  shrinking target problem on the homogeneous space $X_d=G/\Gamma$. %
Before doing so, we first prove a simple reduction lemma which allows us to also assume a lower bound on $\psi$ in Theorem \ref{thm:improvedirichleconwei}.
\begin{Lem}\label{lem:redlem1}
When proving Theorem \ref{thm:improvedirichleconwei}, 
there is no loss of generality to also assume
\begin{align}\label{equ:lowerbdpsi}
\psi(t)\geq \frac{1}{2t}\qquad \text{ for all $t\geq t_0$}.
\end{align}  
\end{Lem}
\begin{proof}
If the given function $\psi$ satisfies $\psi(t)\geq\frac1{2t}$ for all sufficiently large $t$,
then we obtain the desired reduction by simply increasing $t_0$.
It remains to consider the case when $\psi(t)<\frac1{2t}$ holds for an unbounded set of $t$-values. %
Then $F_\psi(t)>\frac12$ 
for an unbounded set of $t$-values, 
and by the argument in Remark \ref{equvserdivREM}, this implies that we are in the divergence case,
i.e.\ our task is to prove that
$\mathrm{Leb}(\DI_{\bm{\alpha},\bm{\beta}}(\psi))=0$.
Now set
\begin{align*}
\psi_1(t):=\max\left\{\psi(t),\tfrac1{2t}\right\}\qquad(t\geq t_0).
\end{align*}
This is a continuous, decreasing function satisfying $\frac1{2t}\leq \psi_1(t)<\frac1t$
and \eqref{equ:con1psinew}.
(To see that $\psi_1$ satisfies \eqref{equ:con1psinew},
note that $F_{\psi_1}(t)=\min\{F_\psi(t),\frac12\}$, from which one easily sees that $F_{\psi_1}$ satisfies \eqref{equ:con1psinew} with the same bounding constant $C_{\psi}$.) 

Now by assumption there is 
an unbounded set of $t$-values
satisfying $\psi(t)<\frac1{2t}$ and thus $F_{\psi_1}(t)=\frac12$. %
By the argument in Remark \ref{equvserdivREM}, this implies that
$\sum_{k}k^{-1}F_{\psi_1}(k)^{\varkappa_d}\log^{\lambda_d}\left(1+\tfrac{1}{F_{\psi_1}(k)}\right)=\infty$.
Assume that Theorem \ref{thm:improvedirichleconwei} has been proved for functions satisfying \eqref{equ:lowerbdpsi}. Applying that restricted statement to $\psi_1$ gives 
$\mathrm{Leb}(\DI_{\bm{\alpha},\bm{\beta}}(\psi_1))=0$.
But $\psi(t)\leq\psi_1(t)$ for all $t\geq t_0$;
hence $\DI_{\bm{\alpha},\bm{\beta}}(\psi)\subset\DI_{\bm{\alpha},\bm{\beta}}(\psi_1)$,
and so $\mathrm{Leb}(\DI_{\bm{\alpha},\bm{\beta}}(\psi))=0$.
\end{proof}
\begin{remark}\label{rmk:comest}
Let $\psi$ be as in Theorem \ref{thm:improvedirichleconwei}. If $\psi$ further satisfies \eqref{equ:lowerbdpsi}, then $F_{\psi}(t)\in (0, \frac12]$ for all $t\geq t_0$. This 
implies %
that 
\begin{align}\label{equ:equvdiverserefpsi}
\sum_{k\geq t_0}k^{-1}F_{\psi}(k)^{\varkappa_d}\log^{\lambda_d}\left(1+\tfrac{1}{F_{\psi}(k)}\right)=\infty \quad \Leftrightarrow\quad \sum_{k\geq t_0}k^{-1}F_{\psi}(k)^{\varkappa_d}\log^{\lambda_d}\left(\tfrac{1}{F_{\psi}(k)}\right)=\infty.
\end{align}
\end{remark}
We now state the Dani correspondence that works for our setting, which in particular incorporates the condition \eqref{equ:con1psinew}. 
\begin{Prop}\label{prop:dyinput}
Let $(\bm{\alpha},\bm{\beta})\in\R^m\times\R^n$ and $\psi: [t_0,\infty)\to (0,\infty)$ be as in Theorem \ref{thm:improvedirichleconwei}, and set $F_{\psi}(t)=1-t\psi(t)$ as before.
We further assume that $\psi$ satisfies \eqref{equ:lowerbdpsi}. Let $\{g_s\}_{s\in \R}$
be the one-parameter diagonal subgroup associated to  $(\bm{\alpha},\bm{\beta})$, defined in \eqref{equ:geneweflow}.
Set
$$\omega_1:=\max\left\{m\alpha_i, n\beta_j\col 1\leq i\leq m, 1\leq j\leq n\right\}\quad\text{and}\quad\omega_2:=\min\left\{m\alpha_i, n\beta_j\col 1\leq i\leq m, 1\leq j\leq n\right\}.$$
Then there exists a  continuous  function $r=r_{\psi}: [s_0,\infty)\to (0, \frac{1}{d})$ with $s_0=\frac{m}{d}\log t_0-\frac{n}{d}\log\psi(t_0)$ 
such that the following properties hold:
\begin{align}\label{equ:qusidecr}
\exists\ C_r\geq 1\ \text{such that}\ r(s_2)\leq C_rr(s_1),\qquad \forall\, 
s_0\leq  s_1\leq s_2\leq s_1+1,
\end{align}
and
for any fixed $u,v>0$,
\begin{align}\label{equ:diequiv}
{\sum_{k\geq t_0}k^{-1}F_{\psi}(k)^{u}\log^{v}\Bigl(\tfrac{1}{F_{\psi}(k)}\Bigr)=\infty}
\quad \Leftrightarrow\quad \sum_{k\geq s_0}r(k)^{u}\log^{v}\Bigl(\tfrac{1}{r(k)}\Bigr)
=\infty.
\end{align}
Moreover, for any $A\in \mathrm{M}_{m,n}(\R/\Z)$, 
\begin{itemize}
\item[(i)] if $\Delta(g_s\Lambda_A)>\omega_1r(s)$ for all sufficiently large $s$, then $A$ is $\psi_{\bm{\alpha},\bm{\beta}}$-Dirichlet;
\item[(ii)] if $\Delta(g_s\Lambda_A)\leq \omega_2r(s)$ for an unbounded set of $s$, then $A$ is not $\psi_{\bm{\alpha},\bm{\beta}}$-Dirichlet.
\end{itemize} 
Here $\Delta: X_d\to [0,\infty)$ is the function given in \eqref{def:delta} and $\Lambda_A=\left(\begin{smallmatrix}
I_m & A\\
0 & I_n\end{smallmatrix}\right)\Z^d$ is as in \eqref{equ:submainfold}. 
\end{Prop}
\begin{proof}
For $\psi$ as in Theorem \ref{thm:improvedirichleconwei}, let $r=r_{\psi}$ be the corresponding function as in \cite[Lemma 8.3]{KleinbockMargulis1999}; it is uniquely determined by $\psi$ via the relation
\begin{align}\label{equ:keyrelation}
\psi(t)^{1/m}e^{s/m}=t^{1/n}e^{-s/n}=e^{-r(s)}.
\end{align}
Solving these equations we get %
\begin{align}\label{equ:keyrel}
e^{-dr(s)}=t\psi(t),\qquad s=\frac{m}{d}\log t-\frac{n}{d}\log\psi(t),\quad \textrm{and}\quad t=e^{s-nr(s)}.
\end{align}
In view of the above first relation and the assumption $\frac12\leq t\psi(t)<1$ (\eqref{equ:con2psi} and \eqref{equ:lowerbdpsi}), we have $r(s)\in (0, \frac{\log 2}{d}]\subset (0,\frac1d)$ for all $s\geq s_0$. Then again by \eqref{equ:keyrel} we have,
with $t=t(s)=e^{s-nr(s)}$:
\begin{align}\label{divequivpf1}
\frac d2r(s)< F_{\psi}(t) <d r(s)\quad\text{and }\quad e^{s-1}<t<e^s,\qquad\forall\ s\geq s_0.
\end{align}
Using also $r(s)\in (0,\frac{\log2}d]$ 
and $F_{\psi}(t)\in (0,\frac12]$, it follows that
$$
r(s)^u\log^v\bigl(\tfrac1{r(s)}\bigr)\asymp_{d,u,v} 
F_{\psi}(t)^{u}\log^{v}\bigl(\tfrac{1}{F_{\psi}(t)}\bigr)\quad \text{
for all $s\geq s_0$ and with $t=t(s)$}.
$$
Hence, using also $e^{k-1}<t(k)<e^k$,
the assumption \eqref{equ:con1psinew} (used repeatedly),
and $\sum_{e^k\leq j<e^{k+1}}\frac1j\asymp1$ ($\forall k\geq1$),
we have for all sufficiently large integers $k$:
\begin{align*}
r(k)^u\log^v\left(\tfrac1{r(k)}\right)
\ll_{d,u,v, C_{\psi}} F_\psi(e^{k-1})^u\log^v\left(\tfrac1{F_\psi(e^{k-1})}\right)
\ll_{u,v, C_{\psi}}\sum_{e^{k-2}\leq j<e^{k-1}}\frac1jF_\psi(j)^u\log^v\left(\tfrac1{F_\psi(j)}\right),
\end{align*}
and similarly
\begin{align*}
r(k)^u\log^v\left(\tfrac1{r(k)}\right)
\gg_{d,u,v, C_{\psi}} F_\psi(e^k)^u\log^v\left(\tfrac1{F_\psi(e^k)}\right)
\gg_{u,v, C_{\psi}}\sum_{e^k\leq j<e^{k+1}}\frac1jF_\psi(j)^u\log^v\left(\tfrac1{F_\psi(j)}\right).
\end{align*}
This shows the equivalence of the divergence of the two series in \eqref{equ:diequiv}.

Next, we show that $r$ satisfies \eqref{equ:qusidecr}. 
Thus let $s_1,s_2$ be given, subject to $s_0\leq s_1\leq s_2\leq s_1+1$.
Set $t_1=t(s_1)$ and $t_2=t(s_2)$.
Since $t(s)$ is an increasing function, we have $t_1\leq t_2$,
and on the other hand,
\begin{align*}
\frac{t_2}{t_1}=e^{s_2-nr(s_2)-s_1+nr(s_1)}
\leq e^{1+nr(s_1)}<e^{1+\frac nd}.
\end{align*}
Hence by \eqref{equ:con1psinew} applied at most 
$N:=\lceil{\frac{1}{\log 2}(1+\frac{n}{d})}\rceil$ times,
$F_{\psi}(t_2)\leq C_{\psi}^N F_{\psi}(t_1)$.
Since also $\frac d2r(s_2)<F_{\psi}(t_2)$ and $F_{\psi}(t_1)<dr(s_1)$,
by \eqref{divequivpf1},
it follows that
$r(s_2)<2C_{\psi}^N r(s_1)$.
We have thus proved that $r$
satisfies \eqref{equ:qusidecr} with $C_r=2C_{\psi}^N$.

The assertion in the ``moreover" part follows from \cite[Proposition 6.2]{KleinbockStrombergssonYu2022}. 
Note that although \cite[Proposition 6.2]{KleinbockStrombergssonYu2022} was stated under the assumption that $t\psi(t)$ is increasing (or equivalently, $r(s)$ is decreasing), this assumption was not used in the proof. 
\end{proof}
In view of Proposition \ref{prop:dyinput} we have the following lemma relating  $\mathbf{DI}_{\bm\alpha,\bm\beta}(\psi)$ to certain limsup sets.
\begin{Lem}\label{lem:dyred}
Let $\psi$ and $r=r_{\psi}$ be as in Proposition \ref{prop:dyinput}, and
let $C_r\geq1$ be a constant such that \eqref{equ:qusidecr} holds.
Define
\begin{align}\label{equ:thickening}
\widetilde{\Delta}_r:=\bigcup_{0\leq s<1}g_{-s}\Delta^{-1}[0, r],\qquad \forall\, r>0. 
\end{align}
For any integer $k\geq s_0$, let 
\begin{align}\label{equ:shrikingupp}
\overline{E}_k:=\left\{A\in \mathrm{M}_{m,n}(\R/\Z): g_k\Lambda_A\in \widetilde{\Delta}_{\omega_1C_{r}r(k)}\right\},
\end{align}
and
\begin{align}\label{equ:shriking}
\underline{E}_k:=\left\{A\in \mathrm{M}_{m,n}(\R/\Z): g_k\Lambda_A\in \widetilde{\Delta}_{\omega_2C_{r}^{-1}r(k+1)}\right\}.
\end{align}
Then  
\begin{align}\label{equ:diinclurela}
\limsup_{k\to\infty}\underline{E}_k\subset \mathbf{DI}^c_{\bm\alpha,\bm\beta}(\psi)\subset \limsup_{k\to\infty}\overline{E}_k.
\end{align}
\end{Lem}
\begin{proof}
For the first inclusion  in \eqref{equ:diinclurela}, take any $A\in \limsup_{k\to\infty}\underline{E}_k$.  In view of the definition of $\widetilde{\Delta}_r$, this implies that %
there exist infinitely many 
$k\geq s_0$ such that the following holds for some $s_k\in[0,1)$:
\begin{align*}
g_{k+s_k}\Lambda_A\in \Delta^{-1}[0, \omega_2C_r^{-1}r(k+1)].
\end{align*}
For $k$ sufficiently large we also have, by \eqref{equ:qusidecr},
\begin{align*}
\Delta^{-1}[0, \omega_2C_r^{-1}r(k+1)]\subset \Delta^{-1}[0, \omega_2r(k+s_k)]. 
\end{align*}
Hence we conclude that there exists an unbounded set of $s$ such that $\Delta(g_s\Lambda_A)\leq \omega_2r(s)$.
Then by item (ii) of the moreover part of Proposition  \ref{prop:dyinput} we have $A\in \mathbf{DI}^c_{\bm\alpha,\bm\beta}(\psi)$. This proves the first inclusion in \eqref{equ:diinclurela}. The second inclusion in \eqref{equ:diinclurela} follows similarly using \eqref{equ:qusidecr} and item (i) of the moreover part of Proposition  \ref{prop:dyinput}.
\end{proof}

\subsection{Measure estimates and smooth approximations}\label{smoothapprSEC}
 In this section we define our new shrinking targets and prove necessary measure estimates and smooth approximation results for these target sets. 

For each $0<r<1$, 
recall that
\begin{align*}
\Delta_r':=\Delta^{-1}[0,r]\cap \cH(R_r),\qquad \text{with $R_r=(1-\tfrac{r}{2d}, 1+\tfrac{r}{2d})\times (-\sqrt{r}, \sqrt{r})^{d-1}$}. 
\end{align*}
The shrinking targets we will be working with are %
the following thickenings of the sets $\Delta_r'$ under the action of $g_s$: %
For any $0<r<1$, define
\begin{align}\label{def:tildedeltar'}
\widetilde{\Delta}_r':=\bigcup_{0\leq s<1/2}g_{-s}\Delta_r'.
\end{align}
The main goal of this section is to prove the following two results regarding the  measure 
and smooth approximation   of $\widetilde{\Delta}'_r$. 

First we have the following measure estimate of $\widetilde{\Delta}_r'$ which is analogous to \cite[Theorem 5.1]{KleinbockStrombergssonYu2022}. 
\begin{Prop}\label{prop:meaest}
Let $\varkappa_d=\tfrac{d^2+d-4}{2}$ and  $\lambda_d=\tfrac{d(d-1)}{2}$ be as in Theorem \ref{thm:improvedirichleconwei}. Then we have
\begin{align}\label{equ:meaest}
\mu_d(\widetilde{\Delta}_r')\asymp_{d} r^{\varkappa_d}\log^{\lambda_d}\Bigl(\tfrac{1}{r}\Bigr),%
\qquad \text{as $r\to 0^+$}. 
\end{align}
\end{Prop}
\begin{proof}
Take $0<r<c^2_0/d^{2d-2}$ as in the proof of Proposition \ref{prop:meaestlobd}. 
Since $\widetilde{\Delta}'_r\subset \widetilde{\Delta}_r$, by \cite[Theorem 5.1]{KleinbockStrombergssonYu2022} we have
$$
\mu_d(\widetilde{\Delta}'_r)\leq \mu_d(\widetilde{\Delta}_r)\asymp_d r^{\varkappa_d}\log^{\lambda_d}\Bigl(\tfrac{1}{r}\Bigr).
$$
We thus only need to prove the lower bound in \eqref{equ:meaest}. For this, let $\underline{K}_r\subset X_d$ be as in Proposition \ref{prop:regulargen} and let $\underline{K}_r'\subset X_d$ be as defined in \eqref{equ:kr'}. 
Recall that $\underline{K}_r'\subset \underline{K}_r$, and by \cite[Lemma 5.2]{KleinbockStrombergssonYu2022}, $\underline{K}_r\cap g_{-s}\underline{K}_r=\emptyset$ for all $s\in [r,1)$. %
Hence we also have
$$
\underline{K}'_r\cap g_{-s}\underline{K}'_r=\emptyset,\qquad \forall\, s\in [r,1).
$$ 
In particular, if we set $\widetilde{\underline{K}}_r':=\bigcup_{0\leq s<1/2}g_{-s}\underline{K}_r'$, then by \eqref{equ:kr'meest} we have (noting also that $\frac{(d-1)(d+2)}{2}=\varkappa_d+1$)
\begin{align}\label{equ:meaestpf1}
\mu_d(\widetilde{\underline{K}}_r')\geq \mu_d\left(\sqcup_{0\leq k<1/2r}g_{-kr}\underline{K}_r'\right)\asymp r^{-1}\mu_d(\underline{K}_r')\gg r^{\varkappa_d}\log^{\lambda_d}\Bigl(\tfrac{1}{r}\Bigr). 
\end{align}
Finally, note that by \eqref{equ:inclukeyALT}, $\underline{K}_{1-e^{-r}}'\subset \Delta'_r$ so that $\widetilde{\underline{K}}_{1-e^{-r}}'\subset \widetilde{\Delta}_r'$. We thus get
\begin{align*}
\mu_d(\widetilde{\Delta}_r')\geq \mu_d(\widetilde{\underline{K}}_{1-e^{-r}}')\gg r^{\varkappa_d}\log^{\lambda_d}\Bigl(\tfrac{1}{r}\Bigr),
\end{align*}
where the last bound holds by \eqref{equ:meaestpf1} and since
$1-e^{-r}\asymp r$.
This finishes the proof. 
\end{proof}

Next we follow the strategy in  \cite[Lemma 6.5]{KleinbockStrombergssonYu2022} to approximate $\widetilde{\Delta}_r'$ by smooth functions and give bounds on the 
$C^{\ell}$-norm of these. 
\begin{Lem}\label{lem:smapp}
For any $0<r<1/16d^2$, 
there exists a smooth function $\phi_r\in C_c^{\infty}(X_d)$ satisfying 
\begin{align}\label{equ:smapp}
\chi_{\widetilde{\Delta}_{r}'}\leq \phi_r\leq \chi_{\widetilde{\Delta}_{2r}'}
\qquad\text{and}\qquad
\|\phi_r\|_{C^\ell}\ll_{\ell} r^{-\ell}\hspace{20pt} (\forall \ell\in\N).
\end{align}
\end{Lem}
\begin{proof}
Set 
$$
\cO_r:=\left\{g\in G: \max\{\|g-I_d\|, \|g^{-1}-I_d\|\}< \tfrac{r}{100d}\right\}. 
$$
Here $\|\cdot\|$ is the supremum norm on $\mathrm{M}_{d,d}(\R)$.
We claim that 
\begin{align}\label{equ:inclre}
\cO_{r_1}\widetilde{\Delta}_{r_2}'\subset \widetilde{\Delta}_{r_1+r_2}',\qquad \forall\, 0<r_1<r_2<1/8d^2.
\end{align}
To prove this, let $0<r_1<r_2<1/8d^2$ be given; it then suffices to show that for any $g\in \cO_{r_1}$,  $0\leq s<\frac12$ and $\Lambda\in \Delta_{r_2}'$, 
$$
g_sgg_{-s}\Lambda\in \Delta_{r_1+r_2}'. 
$$
By the same computation as in the proof of \cite[Lemma 6.5]{KleinbockStrombergssonYu2022} we have $g_{s}gg_{-s}\in \cO_{10r_1}$. Thus it suffices to show
$
\cO_{10r_1}\Delta_{r_2}'\subset \Delta_{r_1+r_2}'
$, and for this it suffices to verify:
$$
\cO_{10r_1}\Delta^{-1}[0, r_2]\subset \Delta^{-1}[0, r_1+r_2]\quad \text{and} \quad \cO_{10r_1}\cH(R_{r_2})\subset \cH(R_{r_1+r_2}). 
$$
The first inclusion above is already proved in the proof of \cite[Lemma 6.5]{KleinbockStrombergssonYu2022}; we thus only need to prove the second one. Given any $g\in \cO_{10r_1}$ and $\Lambda\in \cH(R_{r_2})$, we want to show $g\Lambda\in \cH(R_{r_1+r_2})$. By definition, there exists some $\bm{v}\in \Lambda\cap R_{r_2}$; we will then show $\bm{w}:=g\bm{v}\in R_{r_1+r_2}$, and thereby conclude $g\Lambda \in \cH(R_{r_1+r_2})$. Let $h=g-I_d$; then $\|h\|<\frac{r_1}{10d}$ (since $g\in \cO_{10r_1}$), and $\|\bm{w}-\bm{v}\|=\|h\bm{v}\|$.
Hence 
\begin{align*}
|w_i-v_i|\leq \frac{r_1}{10d}\sum_{i=1}^d|v_i|<\frac{r_1}{10d}\left(\left(1+\frac{r_2}{2d}\right)+(d-1)\sqrt{r_2}\right)<\frac{r_1}{5d},\qquad \forall\, 1\leq i\leq d,
\end{align*} %
where for the last inequality we used the assumption that $r_2<1/8d^2$.
This then implies that  
$$
|w_1-1|<|v_1-1|+\frac{r_1}{5d}<\frac{r_2}{2d}+\frac{r_1}{5d}<\frac{r_1+r_2}{2d},
$$
and 
$$
|w_i|<|v_i|+\frac{r_1}{5d}<\sqrt{r_2}+\frac{r_1}{5d}<\sqrt{r_1+r_2},\qquad \forall\, 2\leq i\leq d.
$$
This shows that $\bm{w}=g\bm{v}\in R_{r_1+r_2}$, finishing the proof of the relation \eqref{equ:inclre}. 

Now by \cite[Lemma 6.6]{KleinbockStrombergssonYu2022} there is some ${\theta}_r\in C_c^\infty(G)$
satisfying 
${\theta}_r\geq0$,
$\supp({\theta}_r)\subset \mathcal{O}_{r/2}$,
$\int_{G}{\theta}_r(g)\,d\nu(g)=1$
and 
$\|\mathcal{D}_Z({\theta}_r)\|_{L^{\infty}(G)}\ll_{\ell_Z,d} r^{1-d^2-\ell_Z}$
for every monomial $Z=Y_1^{\ell_1}\cdots Y_a^{\ell_a}$, where $\ell_Z:=\deg(Z)$. 
Let us define
\begin{align*}
\phi_r(x):=\theta_r* \chi_{\widetilde{\Delta}_{3r/2}'}(x)=\int_G\theta_r(g)\chi_{\widetilde{\Delta}'_{3r/2}}(g^{-1}x)\,\text{d}\nu(g). 
\end{align*}
Then by the same arguments as in the proof of \cite[Lemma 6.5]{KleinbockStrombergssonYu2022} using the relation \eqref{equ:inclre} in place of \cite[(6.12)]{KleinbockStrombergssonYu2022}, we see that $\phi_r$  satisfies the first property in \eqref{equ:smapp}. 
Moreover, again following arguments in the proof of \cite[Lemma 6.5]{KleinbockStrombergssonYu2022},
by using the properties of $\theta_r$ one verifies that for any $x\in X_d$ and any monomial $Z=Y_1^{\ell_1}\cdots Y_a^{\ell_a}$, we have 
$|\cD_Z(\phi_r)(x)|\ll_{\ell_Z, d}r^{-\ell_Z}$. 
Hence $\phi_r$ also satisfies the second property in \eqref{equ:smapp}. 
\end{proof}

\section{Proof of Theorem \ref*{thm:improvedirichleconwei}}
\label{mainthmpfSEC}

In this section we give the proof of Theorem \ref{thm:improvedirichleconwei}. 
Let $(\bm\alpha, \bm\beta)\in \R^{m}\times \R^n$ be the two weight vectors as in Theorem \ref{thm:improvedirichleconwei}. Recall we have set $\omega_1$ and $\omega_2$ as in Proposition \ref{prop:dyinput}. 
Set also 
\begin{align*}
\alpha_{\min}:=\min\{\alpha_i: 1\leq i\leq m\}\quad \text{and}\quad \beta_{\min}:=\min\{\beta_j: 1\leq j\leq n\}.
\end{align*} 
Let $\psi$ be as in Theorem \ref{thm:improvedirichleconwei}, that is, $\psi$ satisfies \eqref{equ:con2psi} and \eqref{equ:con1psinew}. Moreover, in view of Lemma \ref{lem:redlem1}, we may further assume that $\psi$ satisfies \eqref{equ:lowerbdpsi}. 
For such $\psi$, let $r=r_{\psi}$ be as in Proposition \ref{prop:dyinput}. Then by Proposition \ref{prop:dyinput}, $r(s)\in (0,\frac{1}{d})$ for all $s\geq s_0$ and \eqref{equ:qusidecr} holds for some fixed $C_r\geq 1$.

\subsection{Convergence case}\label{convcaseSEC}
In this section we prove the convergence case of Theorem \ref{thm:improvedirichleconwei}. The proof is almost identical to that of the convergence case of \cite[Theorem 1.2]{KleinbockStrombergssonYu2022} with minor modifications because $r$ need not be decreasing and instead satisfies the weaker condition \eqref{equ:qusidecr}. %

Assume that the series in \eqref{equ:zeroonelaw} 
converges. %
First, by  \eqref{equ:equvserdiv}, the second series in \eqref{equ:equvserdiv} also converges. This shows that $\lim_{j\to\infty}F_{\psi}(2^j)=0$, which together with the assumption \eqref{equ:con1psinew} implies that $\lim_{t\to\infty}F_{\psi}(t)=0$, or equivalently, $\lim_{s\to\infty}r(s)=0$. 
Here the latter claim follows from \eqref{divequivpf1}.
Moreover, by \eqref{equ:equvdiverserefpsi} (which applies
since $\lim_{t\to\infty}F_{\psi}(t)=0$) and \eqref{equ:diequiv}, the fact that the series in \eqref{equ:zeroonelaw} converges implies that the series $\sum_k r(k)^{\varkappa_d}\log^{\lambda_d}\bigl(\tfrac{1}{r(k)}\bigr)$ also converges.

Now for any $k>s_0$, let $\rho_k:=\omega_1 C_r r(k)$. Since $\lim_{s\to\infty}r(s)=0$, up to enlarging $s_0$ if necessary we may assume $\rho_k\in (0, \frac12)$ for all $k>s_0$. 
In view of the second inclusion relation in \eqref{equ:diinclurela} and the Borel-Cantelli lemma, it suffices to show $\sum_k\mathrm{Leb}(\overline{E}_k)<\infty$, where for each integer $k\geq s_0$,
\begin{align*}
\overline{E}_k=\left\{A\in \mathrm{M}_{m,n}(\R/\Z): g_k\Lambda_A\in \widetilde{\Delta}_{\rho_k}\right\}.
\end{align*}
For this, note that by \cite[Lemmas 6.5 and 7.1]{KleinbockStrombergssonYu2022} there exists a sequence $\{\phi_{\rho_k}\}_{k>s_0}\subset C_c^{\infty}(X_d)$ satisfying
\begin{align}\label{equ:appabo}
\chi_{\widetilde{\Delta}_{\rho_k}}\leq \phi_{\rho_k},\qquad \forall\, k\geq s_0,
\end{align}
and 
\begin{align}\label{equ:appconv}
\sum_{k>s_0}\int_{\mathcal{Y}}\phi_{\rho_k}(g_k \Lambda_A)\,\mathrm{d} A<\infty,
\end{align}
provided that $\lim_{k\to\infty}\rho_k=0$ and $\sum_k \rho_k^{\varkappa_d}\log^{\lambda_d}\bigl(\tfrac{1}{\rho_k}\bigr)<\infty$. The latter two conditions clearly hold since $\rho_k=\omega_1 C_r r(k)$ and $\{r(k)\}_{k>s_0}$ satisfies 
$\lim_{k\to\infty}r(k)=0$ and $\sum_k r(k)^{\varkappa_d}\log^{\lambda_d}\bigl(\tfrac{1}{r(k)}\bigr)<\infty$. From \eqref{equ:appabo} and \eqref{equ:appconv} we get %
\begin{align*}%
\sum_{k>s_0}\textrm{Leb}(\overline{E}_k)&=\sum_{k>s_0}\int_{\mathcal{Y}}\chi_{\widetilde{\Delta}_{\rho_k}}(g_k\Lambda_A)\,\mathrm{d} A\leq \sum_{k>s_0}\int_{\mathcal{Y}}\phi_{\rho_k}(g_k \Lambda_A)\,\mathrm{d} A<\infty.
\end{align*}
This finishes the proof of the convergence case. 

\subsection{Divergence case}\label{divcaseSEC}

The remainder of this paper is devoted to the proof of the divergence case of Theorem \ref{thm:improvedirichleconwei}. The main work will be to prove the following proposition.
\begin{Prop}\label{keyProp}
Fix $s_0>0$. Let $r$ be a continuous function $r: [s_0,\infty)\to (0,\frac{1}{d})$ satisfying \eqref{equ:qusidecr} %
and 
\begin{align}\label{equ:rdiver}
\sum_{k>s_0}r(k)^{\varkappa_d}\log^{\lambda_d}\bigl(\tfrac{1}{r(k)}\bigr)=\infty.
\end{align}
Then for every $a\in (0,1]$ we have $\mathrm{Leb}(\mathbf{NDI}_{a}(r))=1$, where
\begin{align}\label{DEF:ndir}
\mathbf{NDI}_{a}(r):=\left\{A\in \mathrm{M}_{m,n}(\R/\Z): g_k\Lambda_A\in \widetilde{\Delta}'_{a\, r(k+1)}\ \textrm{for infinitely many }k\in\N
\right\}.
\end{align}
\end{Prop}
\begin{proof}[Proof of the divergence case of Theorem \ref{thm:improvedirichleconwei} assuming Proposition \ref{keyProp}]
Let $\psi$ be as in Theorem~\ref{thm:improvedirichleconwei} and such that the series \eqref{equ:zeroonelaw} diverges.
Recall that %
we may further assume that $\psi$ satisfies \eqref{equ:lowerbdpsi}. 
Set $r=r_{\psi}$.
Now by \eqref{equ:equvdiverserefpsi} and \eqref{equ:diequiv},
the divergence relation \eqref{equ:rdiver} holds.
Let $C_r$ be a bounding constant as in \eqref{equ:qusidecr}, and set $a:=\omega_2 C_r^{-1}\in (0, 1]$.
By the first inclusion relation in \eqref{equ:diinclurela} and the relation $\widetilde{\Delta}_r'\subset \widetilde{\Delta}_r$, we have 
$$
\mathbf{NDI}_a(r)\subset \limsup_{k\to\infty}\underline{E}_k \subset \mathbf{DI}_{\bm\alpha,\bm\beta}^c(\psi)\subset \mathrm{M}_{m,n}(\R/\Z).
$$
Here 
$$
\underline{E}_k=\left\{A\in \mathrm{M}_{m,n}(\R/\Z): g_k\Lambda_A\in \widetilde{\Delta}_{a \, %
r(k+1)}\right\}
$$ 
is  as in \eqref{equ:shriking}. 
But by Proposition \ref{keyProp} we have $\mathrm{Leb}(\mathbf{NDI}_a(r))=1$, which then implies $\mathrm{Leb}(\mathbf{DI}_{\bm\alpha,\bm\beta}^c(\psi))=1$ as desired.
\end{proof}

\subsubsection{A reduction lemma}
The main goal of the remaining sections is now to prove Proposition~\ref{keyProp}. Following the ideas of \cite[Lemmas 12.5, 12.6]{KhalilLuethi2023}, we first prove a reduction  lemma that allows us to further assume that $r(s)$ 
satisfies certain bounds from above and below. %
Let us write
\begin{align}
\fV_d(x)=x^{\varkappa_d}\log^{\lambda_d}(\tfrac{1}{x})
\qquad\text{for }\: 0<x\leq e^{-\lambda_d/\varkappa_d}.
\end{align}
This function $\fV_d$ is continuous, positive, and strictly increasing;
hence the inverse $\fV_d^{-1}$ is a continuous, strictly increasing map
from the interval $\bigl(0,\fV_d(e^{-\lambda_d/\varkappa_d})\bigr]$
onto $(0,e^{-\lambda_d/\varkappa_d}]$.
\begin{Lem}\label{lem:redu}
Fix a parameter
\begin{align}\label{equ:gamdassup}
\gamma_d>\frac{1}{\varkappa_d}. %
\end{align}
When proving Proposition \ref{keyProp}, there is no loss of generality to also assume  %
\begin{align}\label{equ:decayzero}
\lim_{k\to\infty}r(k)=0,
\end{align}
and
\begin{align}\label{equ:reduonr}
r(k)\geq  k^{-\gamma_d},
\qquad \forall\, k\geq s_0.
\end{align}

\end{Lem}

To prove Lemma \ref{lem:redu} we need the following simple auxiliary lemma. 
\begin{Lem}\label{lem:simplelemvar}
Let $\{a_k\}_{k\geq 1}$ be a sequence of positive numbers satisfying $\sum_{k=1}^{\infty}a_k=\infty$.
Then for any $\fV>0$, there exists a decreasing %
sequence of positive numbers $\{b_k\}_{k\geq 1}$ satisfying
\begin{align}\label{equ:claimonbvar}
b_k\leq \fV,\ \forall\, k\geq 1,\quad \lim_{k\to\infty}b_k=0
\quad \text{and}\quad 
\sum_{k=1}^{\infty}\min\{a_k, b_k\}=\infty.
\end{align}
\end{Lem}
\begin{proof}[Proof of Lemma \ref{lem:simplelemvar}]
Since $\sum_{k=1}^{\infty}a_k=\infty$, for any $\e>0$ we have $\sum_{k=1}^{\infty}\min\{a_k, \e\}=\infty$. Thus we can choose a sequence of positive integers $1=N_1<N_2<\cdots$ such that 
\begin{align*}
\sum_{k=N_m}^{N_{m+1}-1}\min\left\{a_k,\tfrac{1}{m}\right\}\geq 1,\qquad \forall\, m\geq 1. 
\end{align*}
Now define $b_k=\min\{\fV, \frac{1}{m}\}$ for any $N_m\leq k<N_{m+1}$. It is clear that $\{b_k\}_{k\geq 1}$ is decreasing and satisfies $b_k\leq \fV$ for all $k\geq 1$ and $\lim_{k\to\infty}b_k=0$. Moreover, 
let $M:=\lceil\frac{1}{\fV} \rceil$ and note that 
\begin{align*}
\sum_{k=1}^{\infty}\min\left\{a_k,b_k\right\}&\geq  \sum_{m=M}^{\infty}\sum_{k=N_m}^{N_{m+1}-1}\min\left\{a_k,\tfrac{1}{m}\right\}\geq \sum_{m=M}^{\infty} 1=\infty. 
\end{align*}
This finishes the proof of the lemma. 
    \end{proof}

\begin{proof}[Proof of Lemma \ref{lem:redu}]
Let $r:[s_0,\infty)\to(0,\frac1d)$ be given as in the statement of 
Proposition \ref{keyProp}; that is, $r$ satisfies
\eqref{equ:qusidecr} and \eqref{equ:rdiver}. %
We first show that we may assume $\lim_{k\to\infty}r(k)=0$. Let us extend the definition of $\fV_d$ by setting $\fV_d(x):=\fV_d(e^{-\lambda_d/\varkappa_d})$
for all $x>e^{-\lambda_d/\varkappa_d}$;
then $\fV_d$ is a continuous and increasing function on all $\R_{>0}$. Let us set $a_k:=\fV_d(r(k))$ (defined for all integers $k\geq s_0$). We claim that 
$\sum_{k\geq s_0}a_k=\infty$. 
For this, note that the function $x\mapsto x^{\varkappa_d}\log^{\lambda_d}(\tfrac{1}{x})$
is decreasing for $e^{-\lambda_d/\varkappa_d}\leq x<1$;
hence we have $x^{\varkappa_d}\log^{\lambda_d}(\tfrac{1}{x})\leq\fV_d(x)$ for all $0<x<1$,
and so it follows from \eqref{equ:rdiver} that
\begin{align}\label{lem:simplelemvarpf1}
\sum_{k>s_0}a_k=\sum_{k>s_0}\fV_d(r(k))=\infty.
\end{align}
Lemma \ref{lem:simplelemvar} (with $\fV=\fV_d(e^{-\lambda_d/\varkappa_d})$) now implies that there exists a decreasing sequence of positive numbers $\{b_k\}_{k>s_0}$ satisfying 
$$
b_k\leq  \fV_d(e^{-\lambda_d/\varkappa_d}) \hspace{10pt} \forall\, k>s_0,\quad  \lim_{k\to\infty}b_k=0\quad \text{ and} \quad \sum_{k>s_0}\min\{a_k, b_k\}=\infty.
$$ 
Since $\fV_d^{-1}$ is a continuous, strictly increasing map
from the interval $\bigl(0,\fV_d(e^{-\lambda_d/\varkappa_d})\bigr]$
onto $(0,e^{-\lambda_d/\varkappa_d}]$, 
setting $\tilde{b}_k:=\fV_d^{-1}(b_k)$ we have that $\{\tilde{b}_k\}_{k>s_0}$ is a
decreasing sequence of positive numbers satisfying $\lim_{k\to\infty}\tilde{b}_k=\lim_{k\to\infty}\fV_d^{-1}( b_k)=0$ and 
\begin{align}\label{equ:r2diveg}
\sum_{k>s_0}\fV_d(\min\{r(k), \tilde{b}_k\})=\sum_{k>s_0}\min\{\fV_d(r(k)), \fV_d(\tilde{b}_k)\}=\sum_{k>s_0}\min\{a_k, b_k\}=\infty.
\end{align}
Since $\{\tilde{b}_k\}_{k>s_0}$ is decreasing, we can choose a continuous decreasing function $\tilde{r}: [s_0,\infty)\to (0, \infty)$ satisfying $\tilde{r}(k)=\tilde{b}_k$ for all integers $k>s_0$.  Define $r_1(s):=\min\{r(s), \tilde{r}(s)\}$. Then one can easily check that $r_1$ is a continuous function from $[s_0,\infty)$ to $(0,\frac1d)$, 
and satisfies \eqref{equ:qusidecr} with the same bounding constant $C_r\geq 1$ (since $\tilde{r}$ is decreasing) and $\lim_{k\to\infty}r_1(k)=0$.
Moreover, note that \eqref{equ:r2diveg} implies that $\sum_{k>s_0}\fV_d(r_1(k))=\infty$.  
Since $r_1(k)\to0$, for all sufficiently large $k$ we have 
$$
\fV_d(r_1(k))=r_1(k)^{\varkappa_d}\log^{\lambda_d}(1/r_1(k)); 
$$
hence $r_1$ satisfies \eqref{equ:rdiver}. Then by Proposition \ref{keyProp}
\textit{with} the extra assumption ``$\lim_{k\to\infty}r(k)=0$'',
we have
$\mathrm{Leb}(\mathbf{NDI}_a(r_1))=1$
for every $a\in(0,1]$.
But clearly $\mathbf{NDI}_a(r_1)\subset\mathbf{NDI}_a(r)$,
since $r_1(s)\leq r(s)$ for all $s$.
Hence also $\mathrm{Leb}(\mathbf{NDI}_a(r))=1$
for every $a\in(0,1]$.

Next, we show that we may also assume $r(k)\geq k^{-\gamma_d}$ for all $k\geq s_0$. 
Thus we assume that $r:[s_0,\infty)\to(0,\frac1d)$ satisfies the conditions
\eqref{equ:qusidecr} and \eqref{equ:rdiver} and also $\lim_{k\to\infty}r(k)=0$.
Now set $r_2(s):=\max\{r(s), s^{-\gamma_d}\}$.
Since %
$\lim_{s\to\infty}s^{-\gamma_d}=0$, up to enlarging $s_0$, we may assume $r_2(s)\in (0,\frac{1}{d})$ for all $s\geq s_0$. Moreover, one easily checks that $r_2(s)$ is continuous, $r_2(k)\to 0$ and satisfies \eqref{equ:qusidecr} with the same bounding constant $C_r\geq 1$. Since $r_2(s)\geq r(s)$, $r_2$ also satisfies \eqref{equ:rdiver}.
Hence by Proposition \ref{keyProp}
with the two extra assumptions 
``$\lim_{k\to\infty}r(k)=0$" and
``$r(k)\geq k^{-\gamma_d}$, $\forall k$'',
we have
\begin{align}\label{LebNDIar1eq1}
\mathrm{Leb}(\mathbf{NDI}_a(r_2))=1,\qquad \forall\, a\in(0,1].
\end{align}
Note also that, since $\gamma_d>\frac{1}{\varkappa_d}$, the series 
$
\sum_k (k^{-\gamma_d})^{\varkappa_d}\log^{\lambda_d}\Bigl(\tfrac{1}{k^{-\gamma_d}}\Bigr)
$
converges. By the same argument as in the proof of the convergence case in Section \ref{convcaseSEC}, %
this implies that for any given number $a\in(0,1]$, the set
\begin{align*}
\left\{A\in \mathrm{M}_{m,n}(\R/\Z): g_k\Lambda_A\in \widetilde{\Delta}_{a(k+1)^{-\gamma_d}}\ \text{for infinitely many $k\in \N$}\right\}
\end{align*}
is of zero Lebesgue measure. Since $\widetilde{\Delta}_r'\subset \widetilde{\Delta}_r$, 
this implies $\mathrm{Leb}(\mathbf{NDI}_a(s^{-\gamma_d}))=0$.
Next we claim that
\begin{align}\label{NDIr1subsetunion1}
\mathbf{NDI}_a(r_2)\subset \mathbf{NDI}_a(r)\:\cup\:\mathbf{NDI}_a(s^{-\gamma_d}).
\end{align}
To prove this, note that if $A\in\mathbf{NDI}_a(r_2)$, then
there is an infinite subset $S\subset\N\cap[s_0,\infty)$
such that
$g_k\Lambda_A\in \widetilde{\Delta}'_{a\,r_2(k+1)}$ for all $k\in S$.
But for each $k$ %
we have either 
$r_2(k+1)=r(k+1)$ or $r_2(k+1)=(k+1)^{-\gamma_d}$;
hence there exists an infinite subset $S'\subset S$
such that either 
$r_2(k+1)=r(k+1)$ for all $k\in S'$, or
$r_2(k+1)=(k+1)^{-\gamma_d}$ for all $k\in S'$.
In the first case we obtain $A\in\mathbf{NDI}_a(r)$,
and in the second case we obtain $A\in\mathbf{NDI}_a(s^{-\gamma_d})$.
Hence \eqref{NDIr1subsetunion1} is proved.
Finally, combining \eqref{LebNDIar1eq1},
$\mathrm{Leb}(\mathbf{NDI}_a(s^{-\gamma_d}))=0$
and \eqref{NDIr1subsetunion1},
we conclude that
$\mathrm{Leb}(\mathbf{NDI}_a(r))=1$.
This completes the proof of the second reduction, and hence also this lemma.
\end{proof}

\subsubsection{The new ingredient: a disjointness fact} %

The main new ingredient for our proof to remove the technical condition \eqref{equ:condipsi} is the following simple disjointness statement concerning $g_s$-translates of $\widetilde{\Delta}'_r$ for certain logarithmic ranges of integer $s$-values. %
\begin{Prop}\label{prop:shortmix}
For any $0<r<  \min\{\frac{d\alpha_{\min}}{2d+1}, \frac{\beta_{\min}}{2}, e^{-4} \}$, 
let $J=J_r:=
\lfloor \frac{1}{4}\log(\frac{1}{r}) \rfloor
$; then for any $k_0\in \Z$, 
the sets 
\begin{align*}
g_{k_0}\widetilde{\Delta}'_{r},\ 
g_{k_0+1}\widetilde{\Delta}'_{r},\ \cdots, \ g_{k_0+J}\widetilde{\Delta}'_{r}
\end{align*}
are pairwise disjoint. 
\end{Prop}
\begin{proof}
For any $0\leq i<j\leq J$, since 
$$
g_{k_0+i}\widetilde{\Delta}_r'\cap g_{k_0+j}\widetilde{\Delta}_r'=g_{k_0+i}(\widetilde{\Delta}'_r\cap g_{j-i}\widetilde{\Delta}_r'),
$$
it suffices to show
\begin{align}\label{prop:shortmixpf1}
\widetilde{\Delta}'_r\cap g_{k}\widetilde{\Delta}_r'=\emptyset,\qquad \forall \,1\leq k\leq J. 
\end{align}
In other words, taking any $1\leq k\leq J$ and $\Lambda\in \widetilde{\Delta}'_{r}$, we want to show $g_{-k}\Lambda\notin \widetilde{\Delta}'_{r}$. 
Suppose, to the contrary, that $g_{-k}\Lambda\in\widetilde{\Delta}'_{r}$. 
Then by definition, there exists some
$s\in [0,1/2)$ such that 
$$
g_{-(k-s)}\Lambda\in \Delta_{r}'=\Delta^{-1}[0,r]\cap \cH(R_{r})\subset  \Delta^{-1}[0, r]. 
$$
This means that 
$$
g_{-(k-s)}\Lambda\cap S_{r}=\{\bm{0}\}\quad \Leftrightarrow\quad \Lambda\cap g_{k-s}S_{r}=\{\bm{0}\},
$$
where $S_r:=(-e^{-r}, e^{-r})^d$.

On the other hand, note that $\Lambda\in \widetilde{\Delta}'_{r}$ means that there exists some $s'\in [0,1/2)$ such that $\Lambda':=g_{s'}\Lambda\in \Delta_{r}'$. This implies that $\Lambda'$ contains a point $\bm{w}\in R_{r}=(1-\frac{r}{2d}, 1+\frac{r}{2d})\times (-\sqrt{r}, \sqrt{r})^{d-1}$. 
Since $\Lambda\cap g_{k-s}S_{r}=\{\bm{0}\}$, we have $\Lambda'\cap g_{k+s'-s}S_{r}=\{\bm{0}\}$. This means that
\begin{align*}
\forall\, \bm{v}\in \Lambda'\setminus\{\bm{0}\}\ :\ |v_i|<e^{\alpha_i(k+s'-s)-r},\ \forall\ 1\leq i\leq m\ \Rightarrow \ \exists\, 1\leq j\leq n\ \text{s.t.}\ |v_{j+m}|\geq e^{-\beta_j(k+s'-s)-r}.
\end{align*}
Since $r<\frac{d\alpha_{\min}}{2d+1}$, we have %
$$
e^{\alpha_i(k+s'-s)-r}>e^{\frac12\alpha_{\min}-r}>1+\frac{r}{2d}>|w_i|,\qquad \forall\, 1\leq i\leq m.
$$ 
Thus 
there must exist some $1\leq j\leq n$ such that $|w_{j+m}|\geq e^{-\beta_j(k+s'-s)-r}$. But since $w_{j+m}\in (-\sqrt{r}, \sqrt{r})$, we have for such $w_{j+m}$,
\begin{align*}
\sqrt{r}>|w_{j+m}|\geq e^{-\beta_j(k+s'-s)-r}\quad\Rightarrow\quad  k>\frac{1}{2\beta_j}\log\Bigl(\tfrac{1}{r}\Bigr)-\frac{r}{\beta_j}+s-s'>\frac{1}{2\beta_j}\log\Bigl(\tfrac{1}{r}\Bigr)-1.
\end{align*}
(For the last inequality we used the assumption that $r<\beta_{\min} /2$.)
The above lower bound on $k$ contradicts our assumption that $k\leq J=\lfloor\frac{1}{4}\log(\frac{1}{r})\rfloor$, since $r<e^{-4}$ and $\beta_j \leq 1$.
This finishes the proof of the proposition.
\end{proof}

\subsubsection{Applications of effective single and double equidistribution  theorems}\label{appleedmtSEC}
In this section we derive various estimates via the effective single and double equidistribution results in Theorem \ref{thm:effequhorsp}. 
Let the function $r$ be as in Proposition \ref{keyProp}; thus $r$ satisfies \eqref{equ:qusidecr} and \eqref{equ:rdiver}. 
In view of Lemma \ref{lem:redu}, we may also assume that $r$ satisfies \eqref{equ:reduonr} with parameter $\gamma_d$ as in \eqref{equ:gamdassup}, and satisfies \eqref{equ:decayzero}. 
Let $\delta>0$ and $\ell\in \N$ be as in Theorem \ref{thm:effequhorsp}. 
Now let us also fix an arbitrary number $0<a \leq  1$.
Our final goal will be to prove that $\mathrm{Leb}(\mathbf{NDI}_{a}(r))=1$
(as in the statement of Proposition \ref{keyProp}). %

In the following, the implicit constant in any ``big-$O$'', ``$\ll$'' or ``$\asymp$'' estimate
will be allowed to depend on $d,\bm{\alpha},\bm{\beta},\gamma_d,\delta,\ell$ and $a$,
without any explicit mention.

For any integer $k\geq s_0-1$, set
\begin{align}\label{rhokDEF}
\rho_k:=\tfrac12 a %
\, r(k+1).
\end{align} 
Recall that $r(k)\to0$ as $k\to+\infty$,
by \eqref{equ:decayzero};
hence for all sufficiently large $k$ we have
\begin{align}\label{rhokbound}
0<\rho_k^{1/2}<\min\Bigl\{\frac1{4d} %
,\frac{d\alpha_{\min}}{4d+2}, \frac{\beta_{\min}}4,\frac12 e^{-4}\Bigr\}.
\end{align}
This inequality ensures that $\rho_k<1/16d^2$ as required in Lemma \ref{lem:smapp}, and also that $r=2\rho_k^{1/2}$ is as small as required in Proposition \ref{prop:shortmix}.
For each such $k$, let $\phi_{\rho_k}\in C_c^{\infty}(X_d)$ be an approximating 
function as in Lemma \ref{lem:smapp}. Then:
\begin{align}\label{EQU:ulbod}
\chi_{\widetilde{\Delta}'_{\rho_k}}\leq \phi_{\rho_k}\leq  \chi_{\widetilde{\Delta}'_{2\rho_k}},
\end{align}
and
\begin{align}\label{EQU:normest}
\|\phi_{\rho_k}\|_{C^{\ell}}\ll \rho_k^{-\ell}.
\end{align}
Note also that \eqref{EQU:ulbod} and Proposition \ref{prop:meaest} imply that for all 
sufficiently large $k$,
\begin{align}\label{EQU:meaestsa}
\text{each of $\:\mu_d(\widetilde{\Delta}'_{\rho_k}),\:$
$\mu_d(\phi_{\rho_k}),\:$
$\mu_d(\widetilde{\Delta}'_{2\rho_k})\:$
is }
\:\asymp \rho_k^{\varkappa_d}\log^{\lambda_d}\Bigl(\tfrac{1}{\rho_k}\Bigr). 
\end{align}

Let us fix a constant $K$ so large that
\eqref{rhokbound} and \eqref{EQU:ulbod}--\eqref{EQU:meaestsa} hold for every $k\geq K$.
(We will impose some further
requirements on $K$ below; however it will be noted that
the constant $K$ can be chosen in a way which only depends on
the function $r$ and $d,\bm{\alpha},\bm{\beta},\gamma_d,\delta,\ell,a$.)
Now for every $k\geq K$,
let us define $h_k\in C_c^{\infty}(X_d)$
and $b_k>0$  through
\begin{align}\label{EQU:bkdef}
h_k(x):=\phi_{\rho_k}(g_kx)   %
\quad\text{and}\quad
b_k:=\int_{\scrY}h_k(\Lambda_A)\,\text{d}A=\int_{\scrY}\phi_{\rho_k}(g_k\Lambda_A)\,\text{d}A.
\end{align}
Also for any $j>i\geq K$, define
\begin{align}\label{EQU:bijdef}
b_{i,j}:=\int_{\scrY}\left(h_i(\Lambda_A)h_j(\Lambda_A)-b_ib_j\right)\,\text{d}A.
\end{align}

We will now deduce
various estimates on $b_{i,j}$ using the double equidistribution 
Theorem \ref{thm:effequhorsp}. 
First, we have by the effective single equidistribution result \eqref{equ:effequ} 
together with \eqref{EQU:normest}:
\begin{align}\label{EQU:eqequilap}
b_k=\mu_d(\phi_{\rho_k})+O(e^{-\delta k}\rho_k^{-\ell})
=\mu_d(\phi_{\rho_k})+O(e^{-\frac{\delta}{2} k}\rho_k^{\varkappa_d}),
\qquad\forall\, k\geq K,
\end{align}
where the last relation holds since
$\rho_k\gg k^{-\gamma_d}$ by \eqref{equ:reduonr}.
Note that because of \eqref{EQU:meaestsa},
the error term in the estimate \eqref{EQU:eqequilap}
is of smaller order of magnitude than the main
term, $\mu_d(\phi_{\rho_k})$,
as $k\to\infty$.
Hence, after possibly increasing the constant $K$, we conclude that
\begin{align}\label{bkasymp}
b_k\asymp\mu_d(\phi_{\rho_k})\asymp\rho_k^{\varkappa_d}\log^{\lambda_d}\Bigl(\tfrac{1}{\rho_k}\Bigr)
\gg \rho_k^{\varkappa_d},
\qquad\forall\, k\geq K.
\end{align}

Next, for any integers $j>i\geq K$, it follows from \eqref{equ:dmix2}, \eqref{EQU:normest}, \eqref{EQU:eqequilap} and \eqref{bkasymp} 
that 
\begin{align}\label{EQU:effmixingcons1new}
\int_{\scrY}h_i(\Lambda_A)h_j(\Lambda_A)\,\text{d}A
&=\bigl(b_i+O(e^{-\frac{\delta}2i}\rho_i^{\varkappa_d})\bigr)\bigl(b_j+O(e^{-\frac{\delta}2j}\rho_j^{\varkappa_d})\bigr)
+O\bigl(e^{-\delta (j-i)}\rho_i^{-\ell}\rho_j^{-\ell}+e^{-\delta i}\rho_i^{-\ell}b_j\bigr).
\end{align}
Let
\begin{align*}
D_{i,j}:=\min\{i, j-i\}.
\end{align*}
Let us also set
\begin{align}\label{Cdef}
{\Theta}:=8\delta^{-1}(\varkappa_d+\ell)\quad \text{and}\quad \vartheta:=\frac{\delta}{4(\ell+\varkappa_d)}. 
\end{align}

Let us first assume 
$$
\rho_j\geq e^{-D_{i,j}/\Theta}, %
$$
or equivalently, $D_{i,j}\geq \Theta\log\bigl(1/\rho_j\bigr)$.
This implies in particular that $\rho_j\geq e^{-(j-i)/\Theta}$.
If we further assume $\rho_j\leq \rho_i e^{\vartheta(j-i)}$,
then
\begin{align}
\rho_i\rho_j\geq\rho_j^2 e^{-\vartheta(j-i)}
\geq e^{-(2/\Theta+\vartheta)(j-i)}.
\end{align}
This implies that 
\begin{align*}
  e^{-\delta (j-i)}\rho_i^{-\ell}\rho_j^{-\ell}  
  \leq e^{(-\delta+(2/\Theta+\vartheta)(\ell+\varkappa_d)) (j-i)} \rho_i^{\varkappa_d}\rho_j^{\varkappa_d}= e^{-\frac{\delta}{2}(j-i)}\rho_i^{\varkappa_d}\rho_j^{\varkappa_d},
\end{align*}
where the last equality holds since
$\bigl(\frac{2}{\Theta}+\vartheta\bigr)(\ell+\varkappa_d)  =\frac{\delta}{2}$,
because of \eqref{Cdef}.
Using also the estimate $e^{-\delta i}\rho_i^{-\ell}\ll e^{-\frac{\delta}{2}i}\rho_i^{\varkappa_d}$,  we have now proved from \eqref{EQU:effmixingcons1new} that    
\begin{align}\label{equ:keyest1}
e^{-D_{i,j}/\Theta}
\leq\rho_j\leq \rho_i e^{\vartheta (j-i)}
\quad\Rightarrow\quad
    |b_{i,j}|\ll e^{-\frac{\delta}{2}D_{i,j}}b_ib_j.
\end{align}

Next, if $\rho_j\geq \rho_i e^{\vartheta(j-i)}$,
then using the trivial bound 
$|b_{i,j}|\leq b_i$ 
we have 
\begin{align*}
    \frac{|b_{i,j}|}{b_j}\leq \frac{b_i}{b_j}\asymp 
    \frac{\rho_i^{\varkappa_d}\log^{\lambda_d}(1/\rho_i)}
{\rho_j^{\varkappa_d}\log^{\lambda_d}(1/\rho_j)}
\ll (\rho_i/\rho_j)^{\varkappa_d/2}
    \leq e^{-\frac{\varkappa_d\vartheta}{2}(j-i)}.
\end{align*}
Hence:
 \begin{align}\label{equ:keyest2}
 \rho_j\geq \max\bigl\{\rho_i e^{\vartheta (j-i)}, %
 e^{-D_{i,j}/\Theta}
 \bigr\}
 \quad\Rightarrow\quad
     |b_{i,j}|\ll e^{-\frac{\varkappa_d\vartheta}{2}(j-i)} b_j.
 \end{align}

 Combining \eqref{equ:keyest1} and \eqref{equ:keyest2} we get
 \begin{align}\label{equ:keyest3new}
D_{i,j}\geq \Theta\log\Bigl(\frac1{\rho_j}\Bigr)
     \quad \Rightarrow\quad
     |b_{i,j}|\ll e^{-c_1D_{i,j}} b_j.
 \end{align}
 Here $c_1:=\frac{\varkappa_d\vartheta}{2}$ (noting that $c_1<\frac{\delta}{2}$). 

 Next, we assume instead that
 \begin{align}
    D_{i,j}\leq \Theta\log\left(\frac{1}{\rho_j}\right).
 \end{align}
First, if $D_{i,j}=i$, then 
\begin{align}\label{jminusilg}
j-i=j-D_{i,j}\geq j-{\Theta}\log\left(\frac{1}{\rho_j}\right)>\frac{3}{4}j,
\end{align}
where the last inequality holds, after possibly increasing  $K$, by using the estimate that $\log(\frac{1}{\rho_j})\ll \log j$ (since $\rho_j\gg j^{-\gamma_d}$). 
Again using $\rho_j,\rho_i\gg j^{-\gamma_d}$, and \eqref{jminusilg}, we have
\begin{align*}
e^{-\delta (j-i)}\rho_i^{-\ell}\rho_j^{-\ell}
<e^{-\frac{\delta}2j}\bigl(e^{-\frac{\delta}8j}\rho_i^{-\ell}\bigr)\bigl(e^{-\frac{\delta}8j}\rho_j^{-\ell}\bigr)
\ll e^{-\frac{\delta}2j}\rho_j^{\varkappa_d}.
\end{align*}
Also $e^{-\delta i}\rho_i^{-\ell}\ll e^{-\frac{\delta}2i}$.
We can thus apply \eqref{EQU:effmixingcons1new} to get in this case
\begin{align*}
b_{i,j}&=\bigl(b_i+O(e^{-\frac{\delta}2i}\rho_i^{\varkappa_d})\bigr)\bigl(b_j+O(e^{-\frac{\delta}2j}\rho_j^{\varkappa_d})\bigr)
+O\bigl(e^{-\frac{\delta}2j}\rho_j^{\varkappa_d}+e^{-\frac{\delta}2i}b_j\bigr)
-b_ib_j=O\bigl(e^{-\frac{\delta}2i}b_j\bigr).
\end{align*}
Hence we conclude that
\begin{align}\label{EQU:smix1new}
D_{i,j}=i\leq {\Theta}\log\left(\frac{1}{\rho_j}\right)\quad\Rightarrow\quad
|b_{i,j}|\ll e^{-\frac{\delta}2 i}b_j.
\end{align}

 Next, if $D_{i,j}=j-i$, let us define the function $\Phi_{i,j}\in C_c^{\infty}(X_d)$ by
\begin{align*}
\Phi_{i,j}(x):=\phi_{\rho_i}(x)\phi_{\rho_j}(g_{j-i}x).
\end{align*}
We then have
 \begin{align*}
\int_{\scrY}h_i(\Lambda_A)h_j(\Lambda_A)\,\text{d}A
&=\int_{\scrY}\phi_{\rho_i}(g_i\Lambda_A)\phi_{\rho_j}(g_{i}g_{j-i}\Lambda_A)\,\text{d}A\\
&=\int_{\scrY}\Phi_{i,j}(g_{i}\Lambda_A)\,\text{d}A=\mu_d(\Phi_{i,j})+O\bigl(\|\Phi_{i,j}\|_{C^{\ell}}\,e^{-\delta i}\bigr),
\end{align*}
where in the last equality we applied \eqref{equ:effequ}.
Using Lemma \ref{lem:sobonorm}, Remark \ref{rmk:operanorm}
and \eqref{EQU:normest},
we conclude that
\begin{align*}
\|\Phi_{i,j}\|_{C^{\ell}}
\ll e^{2\ell(j-i)}\|\phi_{\rho_i}\|_{C^{\ell}}\|\phi_{\rho_j}\|_{C^{\ell}}
\ll e^{2\ell(j-i)}\rho_i^{-\ell}\rho_j^{-\ell}\leq  \rho_i^{-\ell}\rho_j^{-\ell(2\Theta+1)}\ll 
j^{2\gamma_d\ell ({\Theta}+1)}.
\end{align*}
Furthermore, using $i=j-(j-i)\geq j-{\Theta}\log(\frac{1}{\rho_j})> \frac{3}{4}j$
and
$\rho_j\gg j^{-\gamma_d}$ and \eqref{bkasymp},
we have
\begin{align*}
e^{-\delta i}j^{2\gamma_d\ell({\Theta}+1)}< e^{-\frac{3\delta}{4}j}j^{2\gamma_d\ell({\Theta}+1)}
\ll e^{-\frac{\delta}2j}\rho_j^{\varkappa_d}
\ll e^{-\frac{\delta}2j}b_j.
\end{align*}
Hence we conclude that 
\begin{align}\label{EQU:preshoestnew}
D_{i,j}=j-i\leq {\Theta}\log\left(\frac{1}{\rho_j}\right)\quad\Rightarrow\quad
b_{i,j}\leq \int_{\scrY}h_i(\Lambda_A)h_j(\Lambda_A)\,\text{d}A
=\mu_d(\Phi_{i,j})+O\bigl(e^{-\frac{\delta}{2}j}b_j\bigr).
\end{align}

\subsubsection{Proof of Proposition \ref{keyProp}}\label{keyProppfSEC}
We now give the proof of Proposition \ref{keyProp}. 
We continue working in the setup of Section \ref{appleedmtSEC}.
First note that by \eqref{rhokDEF} and \eqref{bkasymp}, 
the divergence assumption \eqref{equ:rdiver} is equivalent to %
\begin{align}\label{EQU:bkdive}
\sum_{k\geq K} b_k=\infty. 
\end{align}

In order to apply the disjointness %
fact from Proposition \ref{prop:shortmix}, we need to ensure that the neighboring targets in the relevant short range are not substantially larger than the target at the time $j$. This is not guaranteed by the quasi-monotonicity condition \eqref{equ:qusidecr}. %
We thus introduce the following notion to single out those times for which  
an appropriate upper bound %
on neighboring targets holds.
For any integer $j\geq K$, we say 
that $j$ is \textit{admissible} %
if 
\begin{align}\label{equ:alphaaddm}
 \rho_i\leq\rho_j^{1/2},\qquad \forall\, i\in\bigl[\max\{K, j-\Theta \,\log(1/\rho_j)\},\,j\bigr)\cap\N.
\end{align}
Denote by $\mathcal{G}\subset [K,\infty)\cap \N$ the set of all admissible indices.

\begin{Lem}
We have %
\begin{align}\label{equ:divercg}
    \sum_{k\in \cG}b_k=\infty.
\end{align}
\end{Lem}
\begin{proof}
Let $\mathcal{B}:=([K,\infty)\cap \N)\setminus \cG$ be the set of non-admissible indices.
    then let $w(j)$ be a fixed choice of such an integer $i$. 
    By \eqref{equ:alphaaddm}, $j\in\cB$ if and only if there exists an integer $i$ with $\max\{K,j-\Theta\log(1/\rho_j)\}\leq i<j$ such that $\rho_i>\rho_j^{1/2}$. For each $j\in\cB$, fix one such integer and denote it by $w(j)$.
    Next, for each $i\geq K$, let $\cS(i)$ be the set of all integers $j\in\cB\cap(i,\infty)$ satisfying
    $w(j)=i$.
    Then for any $K'\geq K$,
    \begin{align}\label{GsumstillinfiniteLEMpf1}
        \sum_{j\in \cB,\: j\leq K'}b_j\leq \sum_{K\leq i<K'}\sum_{j\in\cS(i)}
        b_j\asymp \sum_{K\leq i<K'}\sum_{j\in\cS(i)}
\rho_j^{\varkappa_d}\log^{\lambda_d}\Bigl(\frac1{\rho_j}\Bigr).
    \end{align}
Now for any $i\geq K$ and $j\in\cS(i)$
we have $i<j\leq i+\Theta\log(1/\rho_j)$ and $\rho_j<\rho_i^2$,
and hence for any integer $V\geq0$, if $i\geq K$, $j\in\cS(i)$ and $e^{-V-1}<\rho_j\leq e^{-V}$, then 
$e^{-V-1}<\rho_i^2$, i.e.\ $V>2\log(1/\rho_i)-1$,
and also $i<j %
<i+\Theta(V+1)$.
In particular, the number of indices $j\in\cS(i)$ with $e^{-V-1}<\rho_j\leq e^{-V}$ is $O(V+1)$.
Hence:
\begin{align}\label{GsumstillinfiniteLEMpf2}
\sum_{j\in\cS(i)}\rho_j^{\varkappa_d}\log^{\lambda_d}\Bigl(\frac1{\rho_j}\Bigr)
\ll \sum_{V>2\log(1/\rho_i) -1}  e^{-\varkappa_d V}(V+1)^{\lambda_d+1}
\ll \rho_i^{2\varkappa_d}\log^{\lambda_d+1}\Bigl(\frac1{\rho_i}\Bigr)
\ll b_i \rho_i^{\varkappa_d}\log\Bigl(\frac1{\rho_i}\Bigr),
\end{align}
where the last bound holds by \eqref{bkasymp}.
It follows from \eqref{GsumstillinfiniteLEMpf1} and \eqref{GsumstillinfiniteLEMpf2}
that there is a constant $C>0$ such that for any $K'\geq K$,
\begin{align*}
\sum_{j\in \cB,\: j\leq K'}b_j
\leq C\sum_{K\leq j<K'}b_j\rho_j^{\varkappa_d}\log\Bigl(\frac1{\rho_j}\Bigr).
\end{align*}
But note that $C\rho_j^{\varkappa_d}\log\bigl(\frac1{\rho_j}\bigr)<\frac12$ for all sufficiently large $j$,
since $\lim_{j\to\infty}\rho_j=0$.
Hence there is a constant $C'$ such that for any $K'\geq K$,
\begin{align*}
\sum_{j\in \cB,\: j\leq K'}b_j
\leq C'+\frac12\sum_{K\leq j\leq K'}b_j,
\quad\text{or equivalently,}
\quad
\sum_{j\in \cG,\: j\leq K'}b_j
\geq -C'+\frac12\sum_{K\leq j\leq K'}b_j.
\end{align*}
The lemma follows by letting $K'\to\infty$ in the last inequality.
\end{proof}

Let $\cG\subset [K,\infty)\cap \N$ be the set of admissible integers as above. For any $K'>K$, set
\begin{align*}
Q^{\cG}_{K'}&:=\int_{\mathcal{Y}}\left(\sum_{i\in\cG,\: i\leq K'}h_i(\Lambda_A)-\sum_{i\in\cG,\: i\leq K'}b_i\right)^2\,\text{d} A.
\end{align*}
By a divergence Borel-Cantelli lemma (cf. \cite[Lemma 7.2]{KleinbockStrombergssonYu2022}), since \eqref{equ:divercg} holds, we know that if
\begin{align}\label{EQU:bc}
\liminf_{K'\to\infty}\frac{Q^{\cG}_{K'}}{\left(\sum_{i\in\cG,\: i\leq K'}b_i\right)^2}=0,
\end{align}
then for $\mathrm{Leb}$-a.e.\ $A\in \mathrm{M}_{m,n}(\R/\Z)$, $h_k(\Lambda_A)>0$ for infinitely many $k\in \cG$.
Using $h_k(\Lambda_A)=\phi_{\rho_k}(g_k\Lambda_A)
\leq \chi_{\widetilde{\Delta}'_{2\rho_k}}(g_k\Lambda_A)$
(see \eqref{EQU:bkdef} and \eqref{EQU:ulbod}),
this implies that
$$
\text{for $\mathrm{Leb}$-a.e. $A\in \mathrm{M}_{m,n}(\R/\Z)$,}\quad g_k\Lambda_A\in  \widetilde{\Delta}'_{2\rho_k}
\:
\text{ holds for infinitely many $k\in \cG$}.
$$ 
Hence %
(using $2\rho_k=a\, r(k+1)$   %
and the definition \eqref{DEF:ndir}),
we have $\mathrm{Leb}(\mathbf{NDI}_a(r))=1$,
meaning that Proposition \ref{keyProp} is proved.

We thus only need to prove \eqref{EQU:bc}. 
In fact, since $\sum_{k\in \cG} b_k=\infty$, it suffices to prove the following stronger estimate: 
\begin{align}\label{EQU:keyest}
Q^{\cG}_{K'}\ll \sum_{i\in \cG,\: i\leq K'}b_{i},\qquad  \forall\, K'>K.
\end{align}
We now proceed to prove \eqref{EQU:keyest}. 
By expanding the square in the definition of $Q^{\cG}_{K'}$
and using the definitions of $b_i$ and $b_{i,j}$
(see \eqref{EQU:bkdef} and \eqref{EQU:bijdef}),
we have
\begin{align}\label{Qk1k2bound1}
Q^{\cG}_{K'}&=\sum_{\substack{i,j\in \cG\\i,j\leq K'}}\int_{\cY}(h_i(\Lambda_A)h_j(\Lambda_A)-b_ib_j)\, \mathrm{d}A\leq \sum_{i\in\cG,\: i\leq K'}b_i
+2\sum_{\substack{i,j\in \cG\\i<j\leq K'}}b_{i,j}.
\end{align}
Here for the diagonal terms %
we used the fact that $0\leq h_i\leq1$ to get
$\int_{\cY}(h_i^2-b_i^2)\, \mathrm{d}A
\leq\int_{\cY}h_i\, \mathrm{d}A=b_i$.

Next we treat the double sum over $b_{i,j}$ by
using the
three bounds
\eqref{equ:keyest3new},
\eqref{EQU:smix1new}
and \eqref{EQU:preshoestnew}.
This gives:
\begin{align}\notag
\sum_{\substack{i,j\in \cG\\ i<j\leq K'}}b_{i,j}
\leq
O\biggl(\sum_{j\in\cG,\: j\leq K'}b_j
\sum_{\substack{i\in[K,j)\cap\mathbb{N}\\(D_{i,j}\geq {\Theta}\log(1/\rho_j))}}
e^{-c_1 D_{i,j}}\biggr)
+O\biggl(\sum_{j\in\cG,\: j\leq K'}b_j
\sum_{i\in[K,\Theta\log(1/\rho_j))\cap\mathbb{N}}
e^{-\frac{\delta}2i}\biggr)
\\
+\sum_{j\in\cG,\: j\leq K'}\hspace{1pt}
\sum_{i\in[a_j,j)\cap\mathbb{N}}
\Bigl(\mu_d(\Phi_{i,j})+O\bigl(e^{-\frac{\delta}2 j}\,b_j\bigr)\Bigr),
\end{align}
where $a_j:=\max\{K,j-{\Theta}\log(1/\rho_j)\}$.
Here the three sums 
$\sum_{i\in[K,j)} e^{-c_1 D_{i,j}}$,
$\sum_{i\in[K,\Theta\log(1/\rho_j))} e^{-\frac{\delta}2i}$
and $\sum_{i\in[a_j,j)} e^{-\frac{\delta}2j}$
are all
$\leq 2\sum_{k=1}^{\infty}e^{-c_1 k}\ll1$. 
Hence we conclude:
\begin{align}\label{sumbijbound2}
\sum_{\substack{i,j\in \cG\\ i<j\leq K'}}b_{i,j}
\leq O\biggl(\sum_{j\in\cG,\: j\leq K'}b_j\biggr)
+\sum_{j\in\cG,\: j\leq K'}\hspace{1pt}\sum_{a_j\leq i<j}\mu_d(\Phi_{i,j}).
\end{align}

It remains to bound the double sum in \eqref{sumbijbound2}.
For each $j\in\cG$, define
\begin{align*}
    R_j:=\rho_j^{1/2}.
\end{align*}
Recall that $j\in\cG$ implies that $\rho_i\leq R_j$ for all $i\in[a_j,j)\cap\mathbb{N}$.
Recall also that for any $K\leq i<j$ we have
$\Phi_{i,j}(x)=\phi_{\rho_i}(x)\phi_{\rho_j}(g_{j-i}x)$ and $0\leq \phi_{\rho_i}\leq \chi_{\widetilde{\Delta}'_{2\rho_i}}$.
Hence for each $j\in\cG$, %
\begin{align}\label{Phiijsum1new}
\sum_{a_j\leq i<j}\mu_d(\Phi_{i,j})
\leq \sum_{a_j\leq i<j}\mu_d\bigl(g_{j-i}\widetilde{\Delta}'_{2\rho_i}\cap \widetilde{\Delta}'_{2\rho_j}\bigr)
\leq \sum_{k=1}^{j-a_j} %
\mu_d\bigl(g_k\widetilde{\Delta}'_{2R_j}\cap \widetilde{\Delta}'_{2\rho_j}\bigr),
\end{align}
where the last inequality holds since 
$\widetilde{\Delta}'_{2\rho_i}\subset\widetilde{\Delta}'_{2R_j}$ for all $i$ in the interval $a_j\leq i<j$. %

For each $j\in \cG$, let us now also set
\begin{align*}
J_j:=\Bigl\lfloor\tfrac{1}{4}\log\Bigl(\tfrac{1}{2R_j}\Bigr)\Bigr\rfloor
\:\in\mathbb{N}.
\end{align*}
Then by Proposition \ref{prop:shortmix}
(which applies since for $j\geq K$, $R_j=\rho_j^{1/2}$
satisfies the bound in \eqref{rhokbound} so that 
$r=2R_j$ is as small as required in Proposition \ref{prop:shortmix}),
for any $k\in \N$, the sets 
\begin{align*}
g_{k}\widetilde{\Delta}'_{2R_j}, \ 
g_{k+1}\widetilde{\Delta}'_{2R_j},\ \cdots,\  g_{k+J_j}\widetilde{\Delta}'_{2R_j}
\end{align*}
are pairwise disjoint. 
This implies that the sets
\begin{align*}
g_{k}\widetilde{\Delta}'_{2R_j}\cap \widetilde{\Delta}'_{2\rho_j}, \ 
g_{k+1}\widetilde{\Delta}'_{2R_j}\cap \widetilde{\Delta}'_{2\rho_j},\ \cdots,\  g_{k+J_j}\widetilde{\Delta}'_{2R_j}\cap \widetilde{\Delta}'_{2\rho_j}
\end{align*}
are also pairwise disjoint. 
Since these sets are all contained in $\widetilde{\Delta}'_{2\rho_j}$, 
the sum of their $\mu_d$-measure is bounded from above by $\mu_d(\widetilde\Delta'_{2\rho_j})$.
Hence for each $j\in \cG$: 
\begin{align}\label{mudintsumbound1new}
\sum_{k=1}^{j-a_j} %
\mu_d\bigl(g_k\widetilde{\Delta}'_{2R_j}\cap \widetilde{\Delta}'_{2\rho_j}\bigr)
\leq\Bigl(\tfrac{j-a_j}{J_j+1}+1\Bigr)\mu_d(\tilde\Delta'_{2\rho_j})
\ll \Bigl(\tfrac{\log(1/\rho_j)}{J_j}+1\Bigr)b_j,
\end{align}
where for the last bound we used
$j-a_j\leq \Theta \log(\frac{1}{\rho_j})$ and \eqref{EQU:meaestsa}, \eqref{bkasymp}.
On the other hand, by definition of $R_j$ we have 
$$
J_j\asymp \log\left(\frac{1}{\rho_j}\right).
$$
Using this fact in \eqref{mudintsumbound1new}, we obtain:
\begin{align}
\sum_{k=1}^{j-a_j}%
\mu_d\bigl(g_k\widetilde{\Delta}'_{2R_j}\cap \widetilde{\Delta}'_{2\rho_j}\bigr)
\ll b_j,\qquad \forall\, j\in \cG. %
\end{align}
Combining this bound with
\eqref{Phiijsum1new}, \eqref{sumbijbound2}
and \eqref{Qk1k2bound1},
we conclude that $Q_{K'}^{\cG}\leq O\bigl(\sum_{j\in \cG,\: j\leq K'}b_j\bigr)$.
Since $Q_{K'}^{\cG}$ is non-negative by definition,
this means that we have proved the bound \eqref{EQU:keyest}.
This completes the proof of Proposition \ref{keyProp}. %

\begin{remark}
    Finally, we note that the short-range overlap estimates obtained from Proposition \ref{prop:shortmix} are independent of the initial sampling measure on the expanding horosphere $\cY$. As a result, the  argument extends to other probability measures whenever the corresponding effective single and double equidistribution estimates are available with sufficiently strong and explicit error terms, and the relevant Sobolev norms can be controlled as in Lemma \ref{lem:smapp} or by similar bounds. %
    In particular, when $n=1$ and $\bm{\alpha}=(\frac{1}{m},\cdots, \frac{1}{m})$, one may use \cite[Theorem B and (6.3)]{BHZ2026} in the case $m=1$ and \cite[Theorem 1.2 and (9.3)]{TimotheeHeZhang2025} for general $m$, in place of \eqref{equ:effequ} and \eqref{equ:dmix2} to obtain the same zero-one law as in Theorem \ref{thm:improvedirichleconwei}, but with Lebesgue measure replaced by the fractal measures considered in the corresponding equidistribution results in \cite{BHZ2026,TimotheeHeZhang2025}.
\end{remark}

\bibliographystyle{alpha}
\bibliography{DKbibliog}

\end{document}